\documentclass{article}
\usepackage[a4paper, total={7in, 8in}]{geometry}
\usepackage{enumerate}
\usepackage{amsmath}
\usepackage{amssymb,latexsym}
\usepackage{amsthm}
\usepackage{color}
\usepackage{cancel}
\usepackage{graphicx}
\usepackage{cite}
\usepackage{longtable}
\usepackage{amscd}
\usepackage{lineno}

\makeatletter
\usepackage{amssymb,amsmath,amsthm,latexsym}
\usepackage[mathscr]{eucal} \usepackage{mathrsfs}
\usepackage{color}
\usepackage{hyperref}
\usepackage{comment}
\let\wfs@comment@comment\comment
\let\comment\@undefined

\usepackage{changes}
\let\wfs@changes@comment\comment
\let\comment\@undefined

\newcommand\comment{%
    \ifthenelse{\equal{\@currenvir}{comment}}
    {\wfs@comment@comment}
    {\wfs@changes@comment}%
}
\definechangesauthor[name=Daniele, color=red]{DAN}
\definechangesauthor[name=Giuseppe, color=green]{GIUS}

\makeatother

\def\cC{\mathcal C}

\def\PG{{\rm PG}}

\def\Tr{{\rm Tr}}

\def\F{\mathbb F}

\def\ps@headings{
 \def\@oddhead{\footnotesize\rm\hfill\runningheadodd\hfill\thepage}
 \def\@evenhead{\footnotesize\rm\thepage\hfill\runningheadeven\hfill}
 \def\@oddfoot{}
 \def\@evenfoot{\@oddfoot}
}

\def\Fq3{{\mathbb F}_{q^3}}

\usepackage{comment}

\usepackage{changes}
\begin{document}


\newtheorem{thm}{Theorem}[section]
\newtheorem{lem}[thm]{Lemma}
\newtheorem{prop}[thm]{Proposition}
\newtheorem{cor}[thm]{Corollary}
\newtheorem{Property}[thm]{Property}
\newtheorem*{Main}{Main Theorem}
\newtheorem{Notation}[thm]{Notation}
\newtheorem{conj}[thm]{Conjecture}
\newtheorem{crit}[thm]{Criterion}
\newtheorem{oss}[thm]{Observation}
\newtheorem{question}[thm]{Question}

\newtheorem{definition}[thm]{Definition}
\newtheorem{rem}[thm]{Remark}
\newtheorem{ex}[thm]{Example}

\newtheorem*{main}{Main Theorem}

\newcommand{\fa}[1]{{\color{cyan}{#1}}}

\def\q{{^{q}}}
\def\qq{^{q^2}}
\def\qqq{{^{q^3}}}
\def\qqqq{{^{q^4}}}
\def\qqqqq{{^{q^5}}}
\def\F{\mathbb{F}}

\newcommand{\fax}[2]{{#1^{q^2} +#2^{q}}}
\newcommand{\fbx}[2]{{#1^{q} +#2^{q^3}}}
\title{A new family of $2$-scattered subspaces and related MRD codes}

\author{Daniele Bartoli\thanks{Dipartimento di Matematica e Informatica, Universit\`a degli Studi di Perugia,  Perugia, Italy. daniele.bartoli@unipg.it},  Francesco Ghiandoni\thanks{Dipartimento di Matematica e Informatica ``Ulisse Dini", Universit\`a  degli studi di Firenze, Firenze, Italy, francesco.ghiandoni@unifi.it},
Alessandro Giannoni\thanks{Dipartimento di Matematica e Applicazioni  ``R. Caccioppoli'', Università di Napoli Federico II, Napoli, Italy, alessandro.giannoni@unina.it},
Giuseppe Marino\thanks{Dipartimento di Matematica e Applicazioni  ``R. Caccioppoli'', Università di Napoli Federico II, Napoli, Italy, giuseppe.marino@unina.it}}

\date{}

\maketitle
\begin{abstract}
Scattered subspaces and $h$-scattered subspaces have been extensively studied in recent decades for both theoretical purposes and their connections to various applications. While numerous constructions of scattered subspaces exist, relatively few are known about $h$-scattered subspaces with $h\geq2$. In this paper, we establish the existence of maximum $2$-scattered $\F_q$-subspaces in $V(r,q^6)$ whenever $r\geq 3$, $r\ne 5$, and $q$ is an odd power of $2$. Additionally, we explore the corresponding MRD codes.
\end{abstract}
\section{Introduction}

Linear subspaces have been extensively investigated in recent decades due to their applications in various areas of mathematics, such as finite geometries (blocking sets, two-intersection sets, complete arcs, and caps in affine and projective spaces over finite fields, as well as finite semifields) and coding theory (two weight codes, MRD codes). Among these, scattered subspaces are the most thoroughly studied.

Let us recall some basic definitions on linear sets first. Let $q$ be a prime power and $r,n\in \mathbb{N}$. Let $V$ be a vector space of dimension $r$ over $\mathbb{F}_{q^n}$. For any $k$-dimensional $\mathbb{F}_q$-vector subspace $U$ of $V$, the set $L(U)$ defined by the non-zero vectors of $U$ is called an $\mathbb{F}_q$-\emph{linear set} of $\Lambda=\mathrm{PG}(V, q^n)$ of \emph{rank} $k$, i.e.
\[ L(U)=\{\langle {\mathbf u} \rangle_{\mathbb{F}_{q^n}}: {\mathbf u} \in U\setminus \{\mathbf{0} \}  \}.\]
It is notable that the same linear set can be defined by different vector subspaces. Consequently, we always consider a linear set and the $\mathbb{F}_q$-vector subspace defining it simultaneously. 

Let $\Omega=\mathrm{PG}(W,\mathbb{F}_{q^n})$ be a subspace of $\Lambda$ and let $L(U)$ be an $\mathbb{F}_q$-linear set of $\Lambda$. We say that $\Omega$ has \emph{weight} $i$ in $L(U)$ if $\dim_{\mathbb{F}_q}(W\cap U)=i$. Thus a point of $\Lambda$ belongs to $L(U)$ if and only if it has weight at least $1$. Moreover, for any $\mathbb{F}_q$-linear set $L(U)$ of rank $k$, 
\[|L(U)|\leq \frac{q^{k}-1}{q-1}.\]
When the equality holds, i.e.\ all the points of $L(U)$ have weight $1$, we call $L(U)$ a \emph{scattered} linear set and $U$ a \emph{scattered} $\F_q$-subpace of $V$. A scattered $\mathbb{F}_q$-linear set $L(U)$ of highest possible rank is called a \emph{maximum scattered} $\mathbb{F}_q$-\emph{linear set} (and $U$ is said to be a \emph{maximum scattered} $\F_q$-subspace). 

Recently, scattered and maximum scattered subspaces have been investigated and constructed via suitable polynomial descriptions. This approach started in  \cite{sheekey_new_2016}.  For every $n$-dimensional $\mathbb{F}_q$-subspace $U$ of $\mathbb{F}_{q^n}\times \mathbb{F}_{q^n}$ there exist a suitable basis of $\mathbb{F}_{q^n}\times \mathbb{F}_{q^n}$ and an $\mathbb{F}_q$-linearized polynomial  $f(x)=\sum A_i x^{q^i} \in \mathbb{F}_{q^n}[x]$ of degree less than $q^n$ such that $U= \{ (x, f(x)) : x\in \mathbb{F}_{q^n}  \}.$ Following this description, maximum scattered linear sets in $\mathrm{PG}(1,q^n)$ can be described via the so-called scattered polynomials; \cite{sheekey_new_2016}. Generalizations of these connections led to the index description of scattered polynomials \cite{BZ2018} and scattered sequences \cite{articolosequenze, articolosequenze-new}.

In \cite{CsMPZ2019} the following special class of scattered subspaces has been introduced.

\begin{definition}
	Let $V$ be an $r$-dimensional $\F_{q^n}$-vector space.
	An $\F_q$-subspace $U$ of $V$ is called $h$-scattered, $0<h \leq r-1$, if $\langle U \rangle_{\F_{q^n}}=V$ and each $h$-dimensional $\F_{q^n}$-subspace of $V$ meets $U$ in an $\F_q$-subspace of dimension at most $h$. An $h$-scattered subspace of highest possible dimension is called a maximum $h$-scattered subspace.
\end{definition}

With this definition, the $1$-scattered subspaces are the scattered subspaces generating $V$ over $\F_{q^n}$. When $h=r$, the described definition holds true for $n$-dimensional $\F_q$-subspaces of $V$, which define subgeometries of $\PG(V,\F_{q^n})$.
If $h=r-1$ and $\dim_{\F_q} U=n$, then $U$ defines a scattered $\F_q$-linear set with respect to hyperplanes, introduced in \cite[Definition 14]{ShVdV}.
A further generalisation of the concept of $h$-scattered subspaces can be found in the recent paper \cite{BaCsMT2020}.

In \cite[Theorem 2.3]{CsMPZ2019} it has been proved that for an $h$-scattered subspace $U$ of $V(r,q^n)$, if $U$ does not define a subgeometry, then
\begin{equation}\label{hscatbound}
\dim_{\F_q} U \leq \frac{rn}{h+1}.
\end{equation}
The $h$-scattered subspaces whose dimension reaches Bound \eqref{hscatbound} will be called \emph{maximum} and they exist whenever $h+1 \mid r$; see \cite[Theorem 2.5]{CsMPZ2019}.
Additionally, in \cite[Theorem 2.8]{CsMPZ2019}, it was demonstrated by the authors that $h$-scattered subspaces of dimension $rn/(h+1)$ intersect hyperplanes of $V(r,q^n)$ in $\F_q$-subspaces of dimension ranging from at least $rn/(h+1)-n$ to at most $rn/(h+1)-n+h$. They also introduced a duality relation, referred to as \emph{Delsarte duality}, between maximum $h$-scattered subspaces of $V(r,q^n)$ reaching Bound \eqref{hscatbound} and maximum $(n-h-2)$-scattered subspaces of $V(rn/(h+1)-r,q^n)$ reaching Bound \eqref{hscatbound}. This enabled constructions even when $h+1$ is not a divisor of $r$. Specifically, the authors proved the existence of maximum $(n-3)$-scattered $\F_q$-subspaces of $V(r(n-2)/2,q^n)$ when $n\geq 4$ is even and $r\geq 3$ is odd (cf. \cite[Theorem 3.6]{CsMPZ2019}).

In \cite[Corollary 4.4]{ShVdV} the $(r-1)$-scattered subspaces of $V(r,q^n)$ attaining bound \eqref{hscatbound}, i.e. of dimension $n$, have been shown to be equivalent to MRD-codes of $\F_{q}^{n\times n}$ with minimum rank distance $n-r+1$ and with left or right idealiser isomorphic to $\F_{q^n}$. In \cite{zini2021scattered} a connection between maximum $h$-scattered subspaces and MRD codes has been established.

The main open problem about maximum $h$-scattered in $V(r,q^n)$ is their existence for every admissible values of $r$, $n$, and $h\geq 2$. It is now known that when $2 \mid rn$  there always exist scattered subspaces of maximum dimension \cite{BBL2000, BGMP2015, BL2000, CSMPZ2016}.

Consider the scenario where $h=2$. If $r\equiv 0\pmod 3$, maximum $2$-scattered  subspaces exist for any integer $n$; see \cite[Theorem 2.5]{CsMPZ2019}. For the case when $r=4$ and $n=3$, let $W$ be a maximum $2$-scattered subspace of a hyperplane $H$ in $V(4,q^3)$ (i.e., $W$ is scattered with respect to hyperplanes). Then $W$ has dimension $3$, and let ${\bf u}$ be a vector in $V(4,q^n)$ such that ${\bf u}\notin H$. Consequently, $U:=W\oplus \langle{\bf u}\rangle_{\F_q}$ is a maximum $2$-scattered subspace of $V(4,q^3)$. Considering Theorems 2.5 and 3.6 of \cite{CsMPZ2019}, the first unresolved case pertains to the existence of $2$-scattered $\F_q$-subspaces of $V(4,q^6)$ with dimension $8$.

In this paper we show that they exists, whenever $q=2^h$, with $h\geq 1$ odd.

More precisely, denoted by $\Tr_{{q^n}/q}$ the trace function of $\F_{q^n}$ over
$\F_q$ we prove the following result.

\begin{main}\label{main-thm}
Let $\sigma:x\in\F_{q^6}\mapsto x^{q^{s}}\in\F_{q^6}$ be a field automorphism of $\F_{q^6}$ with $1\leq s\leq 5$ and $\gcd(s,6)=1$. Then the $\F_q$-subspace of $\F_{q^6}^4$
\begin{equation}\label{formU}
U_s:=\{(x,y,x^{\sigma^2}+y^\sigma,x^\sigma+y^{\sigma^3}):x,y\in\F_{q^6}, \Tr_{q^6/q^2}(x)=\Tr_{q^6/q^2}(y)=0\}
\end{equation}
is a maximum $2$-scattered subspace.
\end{main}
As a result, since maximum $2$-scattered $\F_q$-subspaces of dimension $3$ and $4$ exist in $\F_{q^6}$-spaces, as proven in \cite[Theorem 2.5]{CsMPZ2019}, it follows that maximum $2$-scattered $\F_q$-subspaces also exist in $V(r,q^6)$, where $r\geq 3$ and $r\ne 5$, with $q=2^h$ and $h\geq 1$ being odd.

Furthermore, in the last section, we determine the parameters of the associated MRD codes.


\section{Preliminary results}
Since $1\leq s\leq 5$ and $\gcd(s,6)=1$, we have $s\in\{1,5\}$. It can be easily seen that $U_5$ is equivalent to 
\[U'_5=\{(x,y,x^{q^2}+y^q+y^{q^3},x^q+x^{q^3}+y^{q^{3}})\colon \textnormal{ Tr}_{q^6\mid q^2}(x)=\textnormal{ Tr}_{q^6\mid q^2}(y)=0\}\]
and that $U_{1}$ and $U'_5$ are $\mathrm{GL}(4,q^6)$-equivalent. 
Indeed $M\cdot U_5'^T =U_1^T$, where $T$ denotes the transpose and 
\begin{equation*}
        M:=\begin{pmatrix}
            1&0&1&1\\
0&0&1&0\\
1&1&0&1\\
1&0&0&0
        \end{pmatrix}.
    \end{equation*}
From now on, we will denote $U:=U_1$.

\begin{thm}
The $\mathbb{F}_q$-subspace    $U$ is scattered in $V(4,q^6)$.\end{thm}
    \begin{proof}
    Let $T:=\{x\in\F_{q^6}\colon \textnormal{ Tr}_{q^6\mid q^2}(x)=0\}$
 and let $\lambda\in\mathbb{F}_{q^6}\setminus\mathbb{F}_q$ be such that
\begin{equation*}
(x,y,\fax{x}{y},\fbx{x}{y})= \lambda(u,v,\fax{u}{v},\fbx{u}{v}),
\end{equation*}
with $x,y,u,v\in T$. Then, $U$ is scattered if and only if the above equation holds only for $u=v=0$.

By way of contradiction, we assume $(u,v)\neq (0,0)$. We have
\begin{equation} \label{eq sistema U scattered}
\begin{cases}
    x=\lambda u\\
    y=\lambda v\\
    \fax{x}{y}=\lambda(\fax{u}{v})\\
    \fbx{x}{y}=\lambda(\fbx{u}{w})\\
    \textnormal{ Tr}_{q^6\mid q^2}(x)=\textnormal{ Tr}_{q^6\mid q^2}(y)=\textnormal{ Tr}_{q^6\mid q^2}(u)=\textnormal{ Tr}_{q^6\mid q^2}(v)=0,
\end{cases}\end{equation}

so
$$\begin{cases}
    (\lambda+\lambda\qq)u\qq+(\lambda+\lambda\q)v\q=0\\
    (\lambda+\lambda\q)u\q+(\lambda+\lambda\qqq)v\qqq=0\\
    (\lambda+\lambda\qq)u\qq+(\lambda+\lambda\qqqq)u\qqqq=0\\
    (\lambda+\lambda\qq)v\qq+(\lambda+\lambda\qqqq)v\qqqq=0.
\end{cases}$$

It is easy to see that a nontrivial solution $(u,v)$ must satisfy  $uv\neq 0$.
We note that this is a linear system in the unknowns $(\lambda+\lambda\q),(\lambda+\lambda\qq),(\lambda+\lambda\qqq),(\lambda+\lambda\qqqq)$ and, since $(\lambda+\lambda\q)\neq 0$, this is a linear system of 4 equations in 4 unknowns which has a nonzero solution. This is possible if and only if
$\textnormal{det}(M)=0$, where
$$M=\begin{pmatrix}
    v\q&u\qq&0&0\\
    u\q&0&v\qqq&0\\
    0&u\qq&0&u\qqqq\\
    0&v\qq&0&v\qqqq
\end{pmatrix}.$$

We have $\textnormal{det}(M)=-v^{q^3+q}(u\qq v\qqqq+u\qqqq v\qq)$, and it vanishes if and only if $uv\qq+u\qq v=0$, i.e $u=\mu v$ with $\mu\in\F^*_{q^2}.$\\ By replacing $x=\lambda u,$ $y=\lambda v,$ and $u=\mu v$ in the third and fourth equations of System \ref{eq sistema U scattered}, one gets
\begin{equation} 
\begin{cases}
   G(v,v^q,v\qq,v\qqq,\lambda, \lambda\q, \lambda\qq, \lambda\qqq, \lambda\qqqq, \lambda\qqqqq,\mu, \mu\q):= v\q\lambda + v\q\lambda\q + v\qq\lambda \mu + v\qq\lambda\qq\mu=0\\
   H(v,v^q,v\qq,v\qqq,\lambda, \lambda\q, \lambda\qq, \lambda\qqq, \lambda\qqqq, \lambda\qqqqq,\mu, \mu\q) := v\q \lambda \mu\q + v\q \lambda\q \mu\q + v\qqq \lambda + v\qqq \lambda\qqq =0. 
\end{cases}\end{equation}
From $G(v,\lambda,\mu)=0$ we obtain $\lambda\qq=\frac{v\q\lambda + v\q\lambda\q + v\qq\lambda\mu}{v\qq\mu}$ and $\lambda\qqq=\frac{(v\q+v\qq\mu)\lambda + (v\q+v\qq\mu+v\qqq\mu^{1+q})\lambda\q}{v\qqq \mu^{1+q}}$ and likewise $\lambda\qqqq=\Sigma(v,\lambda,\lambda\q,\mu),$ $\lambda\qqqqq=\Theta(v,\lambda,\lambda\q,\mu).$ By replacing such expression of $\lambda\qqq$ in $H(v,\lambda,\mu)=0,$ we have
\begin{equation}
    \frac{\lambda\q+\lambda}{\mu^{1+q}} (v\mu^{2+q} + v + v\q\mu\q + v\qq\mu^{1+q})^q=0.
\end{equation}
So $v\qq=\frac{v\mu^{2+q} + v + v\q\mu\q}{\mu^{1+q}}$ and $v\qqq=\frac{v\mu^{2+q} + v + v\q\mu^{1+3q}}{\mu^{1+2q}}.$ Observe that $\mu=1$ implies $v\in \F_q \subseteq \F_{q^2},$ i.e $v=0$ since $\textnormal{ Tr}_{q^6\mid q^2}(v)=0.$ Thus we can assume $\mu \neq 1.$ After substituting the expressions of $v\qq,v\qqq,\lambda, \lambda\q, \lambda\qq, \lambda\qqq, \lambda\qqqq, \lambda\qqqqq$ in $H^q(v,\lambda,\mu)=0,$ we obtain
\begin{equation}
    \frac{v(\lambda+\lambda\q)}{\mu^{1+q}}(\mu^{2+q}+1)[\mu(\mu^{1+q}(\mu^2+\mu+1)+\mu+1)v+((\mu^{1+q}(\mu\q+\mu+1)+1)v\q]=0.
\end{equation}
\textbf{Case 1.} $\mu^{2+q}+1=0.$\\
From $\mu \in \F_{q^2},$ it follows that $\mu^t=1,$ where $t=\gcd(q^2-1,q+2)=\gcd(q-1,q+2)\in\{1,3\}$. Since $q \equiv -1\pmod 3$, we have $t=1,$ i.e $\mu=1,$ a contradiction. \\
\textbf{Case 2.} $P_\mu v+Q_\mu v\q=0,$ where $P_\mu:=\mu[\mu^{1+q}(\mu^2+\mu+1)+\mu+1]$ and $Q_\mu:=\mu^{1+q}(\mu\q+\mu+1)+1.$
\begin{itemize}
    \item $P_\mu \neq 0 = Q_\mu$ or $Q_\mu \neq 0 = P_\mu.$\\
    This means $v=0,$ a contradiction.
    \item $P_\mu = 0 = Q_\mu.$\\
    Since $\textnormal{Res}(P_\mu,Q_\mu,\mu^q)=\mu^3(\mu^4+\mu^3+1),$ we have $\mu \in \F_{16}\setminus \F_4,$ and a contradiction arises from $\F_{q^2}\cap \F_{16}=\F_4.$ 
\end{itemize}
Thus we can assume without restrictions $P_\mu \neq 0 \neq Q_\mu.$ \\
By replacing $v^q=\frac{P_\mu}{Q_\mu}v$ and the previous expressions of $v\qq,v\qqq,\lambda, \lambda\q, \lambda\qq, \lambda\qqq, \lambda\qqqq, \lambda\qqqqq$ in $H\qq(v,\lambda,\mu)=0$ and in $H\qqq(v,\lambda,\mu)=0,$ we obtain 
$$\begin{cases}
   [(\mu^{1+q})^2+\mu^{1+q}+1]\cdot R_\mu=0\\
   S_\mu=0, 
\end{cases}$$
where $R_\mu:=\mu^{4+4q} + \mu^{4+3q} + \mu^{4+2q} + \mu^{3+4q} + \mu^{2+4q} + \mu^{2+2q} + 1$ and $S_\mu:=\mu^{5+4q} + \mu^{4+2q} + \mu^{2+4q} + \mu^{1+2q} + 1.$\\
We distinguish the following cases.
\begin{itemize}
    \item $(\mu^{1+q})^2+\mu^{1+q}+1=0.$ \\
    It follows that $\mu^{1+q} \in \F_q \cap \F_4=\F_2,$ that is impossible.
    \item $R_\mu=0.$\\
   We observe that
   $$0=R_\mu+S_\mu=\mu^{q}(\mu^{1+q}+1)P_\mu,$$ so $\mu^{1+q}=1.$ Finally, by replacing $\mu^{1+q}=1$ in $S_\mu=0,$ we obtain $(\mu^q+\mu)^2+\mu^q+\mu+1=0,$ and a contradiction arises from $\mu^q+\mu \in \F_q \cap \F_4=\F_2.$
\end{itemize}
\end{proof}

\begin{lem}\label{lemma}
For every subspace $M$ of $\F^4_{q^6}$ fixed by $x\longrightarrow x\qq$ we have  that  ${\dim}_q( U\cap M)$ is even.
\end{lem}
\begin{proof}
   Note that also $U$ is fixed by $x\longrightarrow x\qq.$
    Let $i={\dim}_q( U\cap M)$, so there are $q^i-1$ nonzero vectors in $U\cap M$. Fix $v=(x,y,x\qq+y\q,x\q+y\qqq)\in U\cap M$, $v \neq (0,0,0,0)$; since $U\cap M$ is fixed by $x\longrightarrow x\qq,$  it follows that $v\qq=(x\qq,y\qq,x\qqqq+y\qqq,x\qqq+y\qqqqq)\in U\cap M$. Since $x\neq 0$ or $y\neq 0$ and $\textnormal{ Tr}_{q^6\mid q^2}(x)=\textnormal{Tr}_{q^6\mid q^2}(y)=0$, we have $v\neq v\qq$. Analogously $v\qqqq\in U\cap M$ with $v\neq v\qqqq\neq v\qq$. Thus the orbit of every nonzero vector under the function $x\longrightarrow x\qq$ has size three, hence $3\mid q^i-1$, and that is equivalent to $i$ being even.
\end{proof} 
From now on, we will refer to $i={\dim}_q( U\cap M)$ as the weight of $M$ in $U$. 
\begin{cor}
    For every subspace $M$ of $\F^4_{q^6}$ of dimension $3$ fixed by $x\longrightarrow x\qq$ we have ${\dim}_q( U\cap M)\leq 4$.
\end{cor}
\begin{proof}
    In light of the previous lemma, our task is to demonstrate that the weight of $M$ is at most 5.
    
    Let 
    \begin{equation*}
        ax+by+cz+dt=0,
    \end{equation*}
    with $a,b,c,d\in\F_{q^2}$, be an equation  describing $M$. 
     So we need to count the solutions $(u,v)$ of 
    \begin{equation}\label{w5}
        au+bv+c(u\qq+v\q)+d(u\q+v\qqq)=0,
    \end{equation}
    with $\textnormal{ Tr}_{q^6\mid q^2}(u)=\textnormal{Tr}_{q^6\mid q^2}(v)=0$.
    Observe that if $d=0$, (\ref{w5}) has at most $q^5$ solutions, so we can consider $d=1$.

    We obtain $$
    \begin{cases}
        au+bv+c(u\qq+v\q)+u\q+v\qqq=0\\
        a\q u\q+b\q v\q+c\q(u\qqq+v\qq)+u\qq+v+v\qq=0\\
        a u\qq+b v\qq+c(u+u\qq+v\qqq)+u\qqq+v\q+v\qqq=0.
    \end{cases}$$
    Combining the first and the third equation  we obtain
    \begin{equation}\label{w52}
    \begin{cases}
        (ac+a+c)u+(c+1)u\q+(a+c^2)u\qq+u\qqq+(bc+b)v+(c^2+c+1)v\q+bv\qq=0\\
        a\q u\q+b\q v\q+c\q(u\qqq+v\qq)+u\qq+v+v\qq=0.
    \end{cases}
    \end{equation}
    If $b=0$, considering the first equation,  (\ref{w5}) has at most $q^5$ solutions, so we consider $b\neq 0$. 
    
    Combining the two equations in \eqref{w52} one gets
    \begin{eqnarray}   
    (ac^{q+1}+ac+ac\q+a+c^{q+1}+c)u+(a\q b+c^{q+1}+c+c\q+1)u\q+(ac\q+a+b+c^{q+2}+c^2)u\qq+\nonumber\\\label{w53}(bc\q+c\q+1)u\qqq+b(c^{q+1}+c+c^q)v+(b^{q+1}+c^{q+2}+c^2+c^{q+1}+c+c^q+1)v\q=0.
    \end{eqnarray}
    If this is not the zero polynomial, we can conclude that \eqref{w5} has at most $q^5$ solutions: if the coefficients of $v$ and $v\q$ vanish, then we would have at most $q^3$ solutions $u$, and from the second equation of (\ref{w52}), we would obtain at most $q^5$ solutions $(u,v)$; if they do not vanish, we would obtain the bound directly from Equation (\ref{w53}). 

Thus, we only need to  prove that the following system  has no solutions
    $$\begin{cases}
      ac^{q+1}+ac+ac\q+a+c^{q+1}+c=0\\
      a\q b+c^{q+1}+c+c\q+1=0\\
      ac\q+a+b+c^{q+2}+c^2=0\\
      bc\q+c\q+1=0\\
      c^{q+1}+c+c^q=0\\
      b^{q+1}+c^{q+2}+c^2+c^{q+1}+c+c^q+1=0.
    \end{cases}$$
    From the fourth equation we obtain $b\neq 1$ and $c=\frac{1}{1+b\q}$.

    We can substitute it in the other equations and obtain
     $$\begin{cases}
      b(ab^{q} + 1)=0\\
    b(a\q b^{q+1} + a\q b + a\q b\q + a\q  + b^{q+})=0\\
    b(ab^{q+1} + a+ b^{2q+1} + b + b^{2q})=0\\
    b(b + b^{q} + 1)=0\\
    b(b^{3q+1} + b^{q+1} + b^{3q} + b^{2q} + 1)=0.
    \end{cases}$$
    From the fourth equation we obtain $b^q=b+1$, and substituting it in the fifth equation we obtain 
    \begin{equation*}
        b(b^4+b+1)=0.
    \end{equation*}
    Since $b\neq 0$,  this equation has solution in $\F_{2^4}\setminus\F_{2^2}$. We have $b\in\F_{q^2}\cap\F_{2^4}=\F_{2^2}$, a contradiction.
\end{proof}
\begin{thm}\label{thm3to2}
    A subspace $W$ of dimension $2$ not fixed by $x\longrightarrow x\qq$ contained in a subspace $M$ of dimension $3$ fixed by $x\longrightarrow x\qq$ has weight at most 2.
\end{thm}
\begin{proof}
    Let the weight of $W$ in $U$ be greater than 2. Consider the subspace $W\qq=\{x\qq:x\in W\}\subset U\cap M$. From the assumptions we have $W\neq W\qq$. Since $W,W\qq\subset M$, $W\cap W\qq$ is a subspace of dimension $1$ of weight at least $2$, a contradiction since $U$ is scattered.
\end{proof}

\section{Proof of Main Theorem}
The aim of this section is to prove that $U$ is 2-scattered, hence that for every 2-dimensional subspace of $V(4,q^6)$ $$\begin{cases}
    au+bv+a'w+b't=0\\
    cu+dv+c'w+b't=0,
\end{cases}$$ the system below 
$$\begin{cases}
    a_0u+a_1v+a_2w+a_3t=0\\
    b_0u+b_1v+b_2w+b_3t=0\\
    w=u\qq+v\q\\
    t=u\q+v\qqq\\
    u+u\qq+u\qqqq=0\\
    v+v\qq+v\qqqq=0    
\end{cases}$$
has at most $q^2$ solutions in $u$ and $v$.

Every 2-dimensional subspace of $V(4,q^6)$, say $W$, is equivalent to either
\begin{equation} \label{retta1}\begin{cases}
    au+bv=0\\
    cu+dv=0,        
\end{cases}\end{equation}
with $ad-bc\ne 0$, or
\begin{equation} \label{retta2}\begin{cases}
    u=0\\
    b_1v+b_2w+b_3t=0,
\end{cases}\end{equation}
with $(b_1,b_2,b_3)\ne (0,0,0)$, or
\begin{equation} \label{retta3}\begin{cases}
    au+v=0\\
    b_0u+b_2w+b_3t=0,    
\end{cases}\end{equation}
with $(b_0,b_2,b_3)\ne (0,0,0)$, or\begin{equation} \label{retta4}\begin{cases}
    au+bv+w=0\\
    cu+dv+t=0. 
\end{cases}\end{equation}

Cases \eqref{retta1}  and \eqref{retta2} are trivial. Indeed, System \eqref{retta1} admits the unique solution $(u,v)=(0,0)$, whereas for the subspace $W$ represented by System \eqref{retta2} we have to solve the following
$$\begin{cases}
   u=0\\
   b_1v+b_2w+b_3t=0\\
    w=v\q\\
    t=v\qqq\\
    v+v\qq+v\qqqq=0.    
\end{cases}$$
If $W$, represented by \eqref{retta2} is fixed by $x\longrightarrow x^{q^2}$, then from Lemma \ref{lemma} we get the result. If $W$ is not fixed by $x\longrightarrow x^{q^2}$, since it  
is contained in the 3-dimensional subspace $u=0$, the result follows from Theorem \ref{thm3to2}.

\subsection{Case $W$ represented by \eqref{retta3}}
In this section our aim is to prove that the system below has most $q^2$ solutions $(u,v)$ for any $a,b,c,d\in \mathbb{F}_{q^6}$.
\begin{equation} \label{eq sistema iniziale}\begin{cases}
F_1^{(0)}(u,v) := au+bv+u^{q^2} +v^q=0 \\
F_2^{(0)}(u,v) := cu+dv+u^q+v^{q^3}=0 \\
u+u^{q^2}+u^{q^4}=v+v^{q^2}+v^{q^4}=0.
\end{cases}
\end{equation}

To this end, we will consider the following equations.

\begin{eqnarray}
F_1^{(1)}(u,v) :=& a^qu^q+b^qv^q+u^{q^3} +v^{q^2}&=0\label{F1q}\\   
F_1^{(2)}(u,v) :=& a^{q^2}u^{q^2}+b^{q^2}v^{q^2}+u+u^{q^2} +v^{q^3}&=0\label{F1qq}\\ 
F_1^{(3)}(u,v) :=& a^{q^3}u^{q^3}+b^{q^3}v^{q^3}+u^q+u^{q^3} +v+v^{q^2}&=0\label{F1qqq}\\
F_1^{(4)}(u,v) :=& a^{q^4}(u+u^{q^2})+b^{q^4}(v+v^{q^2})+u +v^q+v^{q^3}&=0\label{F1qqqq}\\
F_1^{(5)}(u,v) :=& a^{q^5}(u^q+u^{q^3})+b^{q^5}(v^q+v^{q^3})+u^q +v&=0\label{F1qqqqq}\\
F_2^{(1)}(u,v) :=& c^qu^q+d^qv^q+u^{q^2}+v+v^{q^2}&=0\label{F2q}\\  
F_2^{(2)}(u,v) :=& c^{q^2}u^{q^2}+d^{q^2}v^{q^2}+u^{q^3}+v^q+v^{q^3}&=0\label{F2qq}\\  
F_2^{(3)}(u,v) :=& c^{q^3}u^{q^3}+d^{q^3}v^{q^3}+u+u^{q^2}+v&=0\label{F2qqq}\\ 
F_2^{(4)}(u,v) :=& c^{q^4}(u+u^{q^2})+d^{q^4}(v+v^{q^2})+u^q+u^{q^3}+v^q&=0\label{F2qqqq}\\ 
F_2^{(5)}(u,v) :=& c^{q^5}(u^q+u^{q^3})+d^{q^5}(v^q+v^{q^3})+u+v^{q^2}&=0.\label{F2qqqqq}
\end{eqnarray}
Also,
$$F_1^{(0)}+F_1^{(2)}+F_1^{(4)}=au+a\qq u\qq+a\qqqq u+a\qqqq u\qq+bv+b\qq v\qq+b\qqqq v+b\qqqq v\qq$$
and 
$$F_2^{(0)}+F_2^{(2)}+F_2^{(4)}=cu+c\qq u\qq+c\qqqq u+c\qqqq u\qq+dv+d\qq v\qq+d\qqqq v+d\qqqq v\qq.$$
Let $\lambda,\mu $ denote $au+a\qq u\qq+a\qqqq u+a\qqqq u\qq$ and $cu+c\qq u\qq+c\qqqq u+c\qqqq u\qq$, respectively. It is readily seen that $\lambda,\mu \in \mathbb{F}_{q^2}$ and $bv+b\qq v\qq+b\qqqq v+b\qqqq v\qq=\lambda$, $dv+d\qq v\qq+d\qqqq v+d\qqqq v\qq=\mu$.

Thus we obtain 
\begin{equation} \label{eqacu}
     (a c\qq +  a c\qqqq +  a\qq c +  a\qq c\qqqq +  a\qqqq c +  a\qqqq c\qq)u + (c+c\qq)^{q^2} \lambda + (a+a\qq)^{q^2} \mu=0
     \end{equation}
and 
 \begin{equation} \label{eqbdv}
     (b d\qq +  b d\qqqq +  b\qq d +  b\qq d\qqqq +  b\qqqq d +  b\qqqq d\qq)v + (d+d\qq)^{q^2} \lambda + (b+b\qq)^{q^2} \mu=0.
     \end{equation}

Let $\alpha,\beta,\gamma$ denote the following  
  \begin{eqnarray}
     \alpha &:=&  a c\qq +  a c\qqqq +  a\qq c +  a\qq c\qqqq +  a\qqqq c +  a\qqqq c\qq \\
     \beta &:=& b d\qq +  b d\qqqq +  b\qq d +  b\qq d\qqqq +  b\qqqq d +  b\qqqq d\qq \\
     \gamma &:=& (b\qq+b\qqqq)(c\qq+c\qqqq)+(d\qq+d\qqqq)(a\qq+a\qqqq).
 \end{eqnarray}
 
 \subsubsection{$\alpha \neq 0$ and $\beta=0$}

\begin{thm}
    If $\alpha\neq 0$, $\beta=0$, $b \in \mathbb{F}_{q^2}$ then System \eqref{eq sistema iniziale} has at most $q^2$ solutions. 
\end{thm}
\begin{proof}
In this case $F_1^{(0)}+F_1^{(2)}+F_1^{(4)}$ reads $au+a\qq u\qq+a\qqqq u+a\qqqq u\qq$. Note that $a,c \notin \mathbb{F}_{q^2}$ otherwise $\alpha=0$. 
Thus $u=\rho (a\qq+a\qqqq)$, for some  $\rho \in \mathbb{F}_{q^2}$. 
Now
$$F_1^{(0)}(\rho (a\qq+a\qqqq),v)=v^q+bv+(a^{q^2+1}+a^{q^4+1}+a+a^{q^4})\rho,$$
from which we get
\begin{eqnarray*}
v^{q^2}&=& ( a^{q^2+1}b\q +  a^{q^4+1}b\q + a b\q+a\qqqq b\q)\rho + (a^{q+q^3} +  a^{q^5+q} + a\q+a\qqqqq)\rho^q  + b^{q+1}v\\
v^{q^4}&=&  (a^{q^2+1} b^{2q+1}+  a^{q^2+1}b\q +  a^{q^4+1} b^{2q+1} + a  b^{2q+1} + 
    a b\q +  a^{q^4+q^2}b\q  + a\qq b\q + a\qqqq  b^{2q+1})\rho \\&&+ (a^{q^3+q}  b^{q+1} + a^{q^3+q} +  a^{q^5+q} b^{q+1} + a\q  b^{q+1}
    + a\q+  a^{q^3+q^5} + a\qqq  
    + a\qqqqq  b^{q+1}) \rho^q   +  b^{2q+2}  v.
\end{eqnarray*}
Recalling that $v+v^{q^2}+v^{q^4}=0$, one gets 
\begin{eqnarray*}( b^{2(q+1)} +  b^{q+1}+1)v= (a^{q^2+1}b^{2q+1} +  a^{q^4+1} b^{2q+1} +  a^{q^4+1}b\q  + a  b^{2q+1}+ 
     a^{q^4+q^2}b\q  + a\qq b\q + a\qqqq  b^{1+2q} + a\qqqq b\q ) \rho \\+ 
    (a^{q+q^3}  b^{q+1} +  a^{q^5+q} b^{q+1} +  a^{q^5+q} + a\q  b^{q+1}  +  a^{q^3+q^5} + a\qqq  + a\qqqqq  b^{q+1} + a\qqqqq )\rho^q.  \end{eqnarray*}
If $( b^{2(q+1)} +  b^{q+1} +1)=0$ we have $b^{q+1}=\omega^i$, $i=1,2$, where $\langle \omega\rangle =\mathbb{F}_4\not\subset \mathbb{F}_q$, a contradiction to $b \in \mathbb{F}_{q^2}$. Hence $v$ directly depends on $\rho$ and thus System \eqref{eq sistema iniziale} has at most $q^2$ solutions.
\end{proof}

\begin{thm}
    If $\alpha\neq 0$, $\beta=0$, $d \in \mathbb{F}_{q^2}$ then System \eqref{eq sistema iniziale} has at most $q^2$ solutions. 
\end{thm}
\begin{proof}
We proceed as in the proof of previous theorem. Recall that $a, c\notin \mathbb{F}_q^{2}$ otherwise $\alpha=0$.

In this case $F_2^{(0)}+F_2^{(2)}+F_2^{(4)}$ reads $cu+c\qq u\qq+c\qqqq u+c\qqqq u\qq=0$, and thus  $u=\rho (c\qq+c\qqqq)$ for some $\rho \in \mathbb{F}_{q^2}$.
Now
$$F_1^{(0)}(\rho (c\qq+c\qqqq),v)=v^q+bv+(a(c^{q^2}+c^{q^4})+c+c^{q^4})\rho,$$
from which we get 
\begin{eqnarray*}
v^{q^3}=a( b^{q^2+q} c\qq + a b^{q+q^2} c\qqqq + a\qq c + 
            a\qq c\qqqq + 
            a\qq c\qqqq + b^{q^2+q} c + b^{q^2+q} c\qqqq + c + c\qq)\rho \\+ (a\q b\qq c\qqq  + a\q b\qq c\qqqqq       + b\qq c\q  + 
    b\qq c\qqqqq )\rho\q +b^{1+q+q^2}v.
\end{eqnarray*}
Since $F_1^{(4)}(\rho (c\qq+c\qqqq),v)=0$, one gets 
\begin{eqnarray} ( b^{q^2+q+1} +  b^{q^4+q+1} +  b +  b\qqqq )v+( a  b^{q^2+q} c\qq +  a  b^{q^2+q} c\qqqq +  a  b^{q^4+q}  c\qq +  a  b^{q^4+q}  c\qqqq +  a  c\qq +  a  c\qqqq + 
         a\qq  c +  a\qq  c\qqqq\nonumber\\ +  a\qqqq  c +  a\qqqq  c\qq +  b^{q^2+q} c +  b^{q^2+q} c\qqqq +  b^{q^4+q}  c + 
         b^{q^4+q}  c\qqqq)\rho+(b + b^{q^2})^{q^2}(a^q(c+c^{q^2})^{q^{3}}+(c+c^{q^2})^{q^{5}})\rho^q=0.\label{Eqr}
\end{eqnarray}
System \eqref{eq sistema iniziale} can be seen now as a linear system in the unknowns $v,\rho,\rho^q$ and this shows that it has at most $q^3$ solutions. 
Note that $b + b^{q^2}\neq 0$ and $a^q(c+c^{q^2})^{q^{3}}+(c+c^{q^2})^{q^{5}}=0$ would imply 
$$a= \frac{(c+c^{q^2})^{q^{4}}}{(c+c^{q^2})^{q^{2}}}, \qquad a^{q^2}= \frac{(c+c^{q^2})}{(c+c^{q^2})^{q^{4}}},\qquad a^{q^4}= \frac{(c+c^{q^2})^{q^4}}{(c+c^{q^2})} $$
and thus $\alpha=0$, a contradiction. Thus, the coefficient of $\rho^q$ in Equation \eqref{Eqr} is nonzero and the total number of solutions of System \eqref{eq sistema iniziale} is at most $q^2$.
\end{proof}

\begin{thm}
 If $\alpha\neq 0$, $\beta=0$, $b,d \notin \mathbb{F}_{q^2}$ then System \eqref{eq sistema iniziale} has at most $q^2$ solutions.     
\end{thm}
\begin{proof}
From Equation \eqref{eqbdv} it follows $$\mu=\frac{d\qq+d\qqqq}{b\qq+b\qqqq}\lambda$$ and thus, by Equation \eqref{eqacu},
   \begin{equation*}
    u:=G(\lambda )= \lambda  \frac{(c\qq +c\qqqq )+(a\qq +a\qqqq ) {( d\qq+ d\qqqq)}/{(b\qq + b\qqqq)}}{(a\qq +a\qqqq ) (c +c\qqqq )+(c\qq +c\qqqq ) (a +a\qqqq )}. 
    \end{equation*}
By $F_1^{(0)}=0$, 
$$v^q=au + bv + u^{q^2}=aG(\lambda) + bv + (G(\lambda))^{q^2}:= H(\lambda,v),$$
and thus in $F_2^{(0)}(G(\lambda),v)=cG(\lambda)+dv+(G(\lambda))^q+(H(\lambda,v))^{q^3}=0$, after clearing the denominators, the coefficient of $v$ is 
$$(b+b^{q^2})^{q^5+q^4+q^3+q^2}(d+b^{q^2+q+1})\alpha^{q+1}.$$
If this coefficient is nonzero, then $v$ can be written in terms of $\lambda \in \mathbb{F}_{q^2}$ and thus System \eqref{eq sistema iniziale} has at most $q^2$ solutions. On the other hand, by our assumptions, such a coefficient vanishes only if $d=b^{q^2+q+1}$. The same argument applies to $F_1^{(1)}(G(\lambda),v)$, whose coefficients in $v$ is 
$$(b+b^{q^2})^{q^5+q^4+q^3+q^2}(b^{q^3+q^2+q+1}+b^{q+1}+1)\alpha^{q+1}.$$
Thus, from now on we may assume $d=b^{q^2+q+1}$ and $b^{q^3+q^2+q+1}+b^{q+1}+1=0$. Note that $b^{q^2+2q+1}+b^{q+1}+1\neq0$ otherwise $b^{q^3+q^2+q+1}+b^{q+1}+1=0$ would imply $b \in \mathbb{F}_{q^2}$, a contradiction.

Imposing that $b^{q^3}=(b^{q+1}+1)/b^{q^2+q+1}$ one gets that, after clearing the denominators, in $F_1^{(3)}(G(\lambda),v)$ the coefficient of $\lambda^q$ is 
$$(b+b^{q^2})^3(b^{q^2+2q+1}+b^{q+1}+1)\alpha L(a,b,c),$$
where 
\begin{eqnarray}
L(a,b,c)&:=&    a^{q^3+q} b^{q^2+2q+1}  + a^{q^3+q} b^{q^2+q} + a^{q^3+q} + a^{q^5+q} b^{q^2+2q+1}  + a^{q^5+q} b^{q^2+q} +
        a^{q^5+q} + a\q b^{q^2+2q+1}  c\qqq+\nonumber\\&&  + a\q b^{q^2+2q+1}  c\qqqqq  + a\q  b^{q^2+2q+1}  + a\q  b^{q+1}  c\qqq  
        + a\q  b^{q+1} c\qqqqq  + a\q  b^{q^2+q} + a\q  c\qqq  + a\q  c\qqqqq  + a\q+\nonumber\\&&  + a\qqqqq  b^{q^2+2q+1}  + a\qqqqq  b^{q^2+q} +
        a\qqqqq  +b^{q^2+2q+1}   c\q  +b^{q^2+2q+1}  c\qqqqq  + b^{q+1}  c\q  + b^{q+1}  c\qqqqq  + c\q  + c\qqqqq .
\end{eqnarray}
If such a coefficient is nonzero we can obtain a nonzero linear equation involving $\lambda$, $\lambda^q$, and $v$ and thus System \eqref{eq sistema iniziale} has at most $q^2$ solutions.   
By direct computations, after clearing the denominators,
$$L(a,b,c)+(L(a,b,c))^{q^2}+(L(a,b,c))^{q^4}=(b^{q^2+2q+1}+b^{q+1}+1)\alpha^q$$
and thus $L(a,b,c)\neq 0$ and System \eqref{eq sistema iniziale} has at most $q^2$ solutions.  
\end{proof}

 \subsubsection{$\alpha =0$ and $\beta \neq 0$}

\begin{thm}
    If $\alpha= 0$, $\beta\neq0$, $a \in \mathbb{F}_{q^2}$ then System \eqref{eq sistema iniziale} has at most $q^2$ solutions. 
\end{thm}
\begin{proof}
In this case $F_1^{(0)}+F_1^{(2)}+F_1^{(4)}$ reads $bv+b\qq v\qq+b\qqqq v+b\qqqq v\qq$. Note that $b,d \notin \mathbb{F}_{q^2}$ otherwise $\beta=0$. 
Thus $v=\rho (b\qq+b\qqqq)$, for some  $\rho \in \mathbb{F}_{q^2}$. Now,
$$F_2^{(0)}(u,\rho (b\qq+b\qqqq))=u^q+cu +(B^q+B^{q^5})\rho^q+d(b^{q^2}+b^{q^4})\rho.$$
Thus, the substitution $u^q=H(u,\rho):=cu +(B^q+B^{q^5})\rho^q+d(b^{q^2}+b^{q^4})\rho$ makes System \eqref{eq sistema iniziale}  a system in $u,\rho,\rho^q$. If one among the coefficients of $u,\rho,\rho^q$ is nonzero, such a system and thus System \eqref{eq sistema iniziale} has at most $q^2$ solutions. 
We consider the coefficients of $v$ in $F_1^{(0)}(u,\rho)$, and $F_1^{(2)}(u,\rho)$ together with the coefficient of $\rho^q$ in $F_1^{(0)}(u,\rho)$ and $F_2^{(4)}(u,\rho)$. If they all vanish then 

\begin{equation}\label{EqSopra}
\begin{cases}
    a + c^{q+1}=0\\
    a  c^{q+1}+  c^{q+1}+ 1=0\\
b\q c\q + b\qqq d\q + b\qqq + b\qqqqq c\q + b\qqqqq d\q + b\qqqqq=0\\
b\q  c^{q^2+q}+ b\q  c^{q^4+q}+ b\qqq c\qq d\q + b\qqq c\qqqq d\q + b\qqqqq  c^{q^2+q}+ b\qqqqq  c^{q^4+q}+ b\qqqqq c\qq d\q + b\qqqqq c\qqqq d\q=0
.\end{cases}
\end{equation}
First note that $c=0$ yields $a=0$, a contradiction to the second equation of System \eqref{EqSopra}. So $c\neq 0$. Also, $a=c^{q+1}$, $c^{q+1}=\omega^{i}$, $i=1,2$. The last equation gives 
$$b\q = \omega^2 b\qqq c d\q + \omega^2 b\qqqqq c d\q + b\qqqqq,$$
and thus $b\q c\q + b\qqq d\q + b\qqq + b\qqqqq c\q + b\qqqqq d\q + b\qqqqq=0$ yields $c(b^{q^3}+b^{q^5})=0$, a contradiction. This means that not all the coefficients in $v,\rho,\rho^q$ can vanish simultaneously and the claim follows.
\end{proof}

\begin{thm}
    If $\alpha= 0$, $\beta\neq0$, $c \in \mathbb{F}_{q^2}$ then System \eqref{eq sistema iniziale} has at most $q^2$ solutions. 
\end{thm}
\begin{proof}
In this case $F_2^{(0)}+F_2^{(2)}+F_2^{(4)}$ reads $dv+d\qq v\qq+d\qqqq v+d\qqqq v\qq$. Note that $b,d \notin \mathbb{F}_{q^2}$ otherwise $\beta=0$. 
Thus $v=\rho (d\qq+d\qqqq)$, for some  $\rho \in \mathbb{F}_{q^2}$. Also, Now
$$F_2^{(0)}(u,\rho (d\qq+d\qqqq))=c u + d^{q^2+1}\rho + d^{q^4+1}\rho + d\q\rho\q + d\qqqqq\rho\q + u^q.$$
It follows that 
\begin{eqnarray*}
    u^{q^2}&=&c^{q+1} u + c\q d^{q^2+1}\rho + c\q d^{q^4+1}\rho + c\q d\q\rho\q + c\q d\qqqqq\rho\q + d\rho + 
     d^{q^3+q}\rho\q +  d^{q^5+q}\rho\q + d\qq\rho\\
    u^{q^4}&=&  c^{2q+2} u +c^{2q+1} d^{q^2+1}\rho + c^{2q+1} d^{q^4+1}\rho + c^{2q+1} d\q\rho\q + 
    c^{2q+1} d\qqqqq\rho\q + c^{q+1} d\rho + c^{q+1}  d^{q^3+q}\rho\q + c^{q+1}  d^{q^5+q}\rho\q + 
    \\&&+c^{q+1} d\qq\rho+ c\q d^{q^2+1}\rho + c\q d\q\rho\q + c\q  d^{q^4+q^2}\rho + c\q d\qqq\rho\q + 
     d^{q^3+q}\rho\q + d\qq\rho + d^{q^5+q^3}\rho\q + d\qqqq\rho.
\end{eqnarray*}
Recalling that $u+u^{q^2}+u^{q^4}=0$ one gets 
\begin{eqnarray*}
       (c^{2q+2} + c^{q+1} + 1)u\!\!\!\!\!&=\!\!\!\!\!& c^{2q+1} d^{q^2+1}\rho + c^{2q+1} d^{q^4+1}\rho + c^{2q+1} d\q\rho\q + c^{2q+1} d\qqqqq\rho\q + 
        c^{q+1} d\rho + c^{q+1}  d^{q^3+q}\rho\q + c^{q+1}  d^{q^5+q}\rho\q +\\&&+ c^{q+1} d\qq\rho + 
        c\q d^{q^4+1}\rho + c\q  d^{q^4+q^2}\rho + c\q d\qqq\rho\q + c\q d\qqqqq\rho\q + d\rho + 
         d^{q^5+q}\rho\q + d^{q^5+q^3}\rho\q + d\qqqq\rho.\\
\end{eqnarray*}
Since $ c^{2q+2} + c^{q+1} + 1=0$ would imply $c^{q+1}=\omega^i$, $i=1,2$, a contradiction to $c \in \mathbb{F}_{q^2}$ (since $\mathbb{F}_4\not\subset \mathbb{F}_q$), $u$ depends on $\rho \in \mathbb{F}_{q^2}$ and thus System \eqref{eq sistema iniziale} has at most $q^2$ solutions. 
\end{proof}

\begin{thm}
    If $\alpha= 0$, $\beta\neq0$, $a,c \notin \mathbb{F}_{q^2}$ then System \eqref{eq sistema iniziale} has at most $q^2$ solutions. 
\end{thm}
\begin{proof}
From Equation \eqref{eqacu} it follows $$
\mu = \frac{(c\qq+c\qqqq)\lambda}{a\qq+a\qqqq},$$ and thus, by Equation \eqref{eqbdv},
   \begin{equation*}
    v=G(\lambda ):= \left((d\qq+d\qqqq)   \lambda + (b\qq+b\qqqq)  \frac{(c\qq+c\qqqq) \lambda}{a\qq+a\qqqq}\right)/\left(b d\qq +  b  d\qqqq +  b\qq  d +  b\qq d\qqqq +  b\qqqq  d +  b\qqqq  d\qq\right).
    \end{equation*}
By $F_2^{(0)}=0$, 
$$u\q=cu + dv  + v\qqq=dG(\lambda) + cu + (G(\lambda))^{q^3}=: H(\lambda,u),$$
and thus in $F_1^{(0)}(u,G(\lambda))=0$, after clearing the denominators, the coefficient of $u$ is 
$a(a+a\qq)^{q^3+q^2}\beta^{q+1}$, which is nonzero.
Therefore, $u$ can be written in terms of $\lambda \in \mathbb{F}_{q^2}$ and thus System \eqref{eq sistema iniziale} has at most $q^2$ solutions. 
\end{proof}
\begin{comment}
    alfa0betanon0a,cnoninfq2.txt
\end{comment}

  \subsubsection{$\alpha=0=\beta,$ $\gamma=0$}\label{sec4}
  We want to prove that $2$-dimensional subspaces of this type are contained in a $3$-dimensional subspace fixed by $x\longrightarrow x\qq$. Note that if $a,b\in\F_{q^2}$ or $c,d\in\F_{q^2}$ then respectively $F_1^{(0)}(u,v)=0$ or $F_2^{(0)}(u,v)=0$ is the subspace we are looking for. 

  \textbf{- Case }$c\notin\F_{q^2}$. Let $\lambda:=(a+a\qq)/(c+c\qq)$. From $\alpha=0$ we obtain $\lambda\in\F_{q^2}$. Then 
  \begin{eqnarray*}
      F_{\lambda}(u,v):&=&F_1^{(0)}(u,v)+F_1^{(2)}(u,v)+\lambda(F_2^{(0)}(u,v)+F_2^{(2)}(u,v))\\&=&(a+c\lambda)u+\lambda u^q+(a\qq+c\qq\lambda)u\qq+\lambda u\qqq+(b+d\lambda)v+(\lambda+1)v\q+(b\qq+d\qq\lambda)v\qq+v\qqq.
  \end{eqnarray*}
  From the definition on $\lambda$ we have $(a+c\lambda)\in\F_{q^2}$. If $d\notin\F_{q^2}$ then from $\beta=0$ and $\gamma=0$ we obtain $\lambda=(b+b\qq)/(d+d\qq)$, and so $(b+d\lambda)\in\F_{q^2}$.
  If $d\in\F_{q^2}$ from $\gamma=0$ we obtain $b\in\F_{q^2}$, and so $(b+d\lambda)\in\F_{q^2}$.
  Hence, in all the  cases $F_{\lambda}(u,v)=0$ defines a $3$-dimensional subspace fixed by $x\longrightarrow x\qq$.

\textbf{- Case }$c\in\F_{q^2}$. We have  $d\notin\F_{q^2}$ and so from $\gamma=0$ we obtain $a\in\F_{q^2}$. Let $\lambda:=(b+b\qq)/(d+d\qq)$ and analogously to the previous case we obtain that $F_{\lambda}(u,v)=0$ defines a $3$-dimensional subspace fixed by $x\longrightarrow x\qq$.
\vspace{0.5cm}

 Note that our objective in this subsection is to establish that every $2$-dimensional subspace $W$ has a weight of at most $3$ in $U$. Specifically, if $W$ is stabilized by the map $x\longrightarrow x\qq$, then it possesses an even weight, thus at most 2. Conversely, if $W$ is not stabilized by $x\longrightarrow x\qq$, we have just demonstrated the existence of a $3$-dimensional subspace containing $W$, which is stabilized by $x\longrightarrow x\qq$. Consequently, we can apply Theorem \ref{thm3to2}.

  \begin{thm}
    If $\alpha=\beta=\gamma= 0$, and one among $a,b,c,d$ is not in $\F_{q^2}$ then System \eqref{eq sistema iniziale} has at most $q^2$ solutions.
\end{thm}
\begin{proof}
Observe that if the $2$-dimensional subspace $W$ defined by $F_1^{(0)}(u,v)=F_2^{(0)}(u,v)=0$ is not fixed by $x\longrightarrow x\qq$ we can apply Theorem \ref{thm3to2}. So we can consider $W$ fixed by $x\longrightarrow x\qq$.
Let us divide the proof in several cases.

\textbf{Case 1.} $a\notin\F_{q^2}$:
we have
\begin{equation*}
    \begin{cases}
        au+bv+u\qq+v\q=0\\
        au\qq+bv\qq+u+u\qq+v\qqq=0\\
        a\qq u\qq+b\qq v\qq+u+u\qq+v\qqq=0\\
        cu+dv+u\q+v\qqq=0.
    \end{cases}
\end{equation*}
Summing the second and third equation we obtain 
\begin{equation*}
    u=\frac{b+b\qqqq}{a+a\qqqq}v,
\end{equation*}
and substituting it into the fourth equation we obtain a bound on the solutions $(u,v)$ of $q^3$. Using Lemma \ref{lemma} the claim follows.

\textbf{Case 2.} $b\notin\F_{q^2}$:
we have
\begin{equation*}
    \begin{cases}
        au+bv+u\qq+v\q=0\\
        au\qq+bv\qq+u+u\qq+v\qqq=0\\
        a\qq u\qq+b\qq v\qq+u+u\qq+v\qqq=0\\
        cu+dv+u\q+v\qqq=0.
    \end{cases}
\end{equation*}
Analogously to the previous case we obtain
\begin{equation*}
    v=\frac{a+a\qqqq}{b+b\qqqq}u,
\end{equation*}
and substituting it into the fourth equation we obtain a bound on the solutions $(u,v)$ of $q^3$. Using Lemma \ref{lemma} the claim follows.

\textbf{Case 3.} $c\notin\F_{q^2}$:
we have
\begin{equation*}
    \begin{cases}
        cu+dv+u\q+v\qqq=0\\
        cu\qq+dv\qq+u\qqq+v\q v\qqq=0\\
        c\qq u\qq+d\qq v\qq+u\qqq+v\q v\qqq=0\\
        au+bv+u\qq+v\q=0.
    \end{cases}
\end{equation*}
Summing the second and third equation we obtain 
\begin{equation*}
    u=\frac{d+d\qqqq}{c+c\qqqq}v,
\end{equation*}
and substituting it into the fourth equation we obtain a bound on the solutions $(u,v)$ of $q^2$.

\textbf{Case 4.} $d\notin\F_{q^2}$:
we have
\begin{equation*}
    \begin{cases}
        cu+dv+u\q+v\qqq=0\\
        cu\qq+dv\qq+u\qqq+v\q v\qqq=0\\
        c\qq u\qq+d\qq v\qq+u\qqq+v\q v\qqq=0\\
        au+bv+u\qq+v\q=0.
    \end{cases}
\end{equation*}
Analogously to the previous case we obtain
\begin{equation*}
    v=\frac{c+c\qqqq}{d+d\qqqq}u,
\end{equation*}
and substituting it into the fourth equation we obtain a bound on the solutions $(u,v)$ of $q^2$.

\end{proof}

 \begin{thm}
    If $\alpha=\beta=\gamma= 0$, and $a,b,c,d\in\F_{q^2}$ then System \eqref{eq sistema iniziale} has at most $q^2$ solutions.
\end{thm}
 \begin{proof}
    Consider 
    \begin{equation*}
        \begin{cases}
        au+bv+u\qq+v\q=0\\
        cu+dv+u\q+v\qqq=0\\
        a^qu^q+b^qv^q+u^{q^3} +v^{q^2}=0\\ 
        c^qu^q+d^qv^q+u^{q^2}+v+v^{q^2}=0\\        
        au^{q^2}+bv^{q^2}+u+u^{q^2} +v^{q^3}=0\\ 
        cu^{q^2}+dv^{q^2}+u^{q^3}+v^q+v^{q^3}=0\\
        a^{q}u^{q^3}+b^{q}v^{q^3}+u^q+u^{q^3} +v+v^{q^2}=0.
        \end{cases}
    \end{equation*}
    From the third, second and first equations, we derive respectively $u\qqq,u\q,u\qq$ which we substitute into the others, obtaining
    \begin{equation*}
        \begin{cases}
        a u + b v + c^{q+1} u + c\q d v + c\q v\qqq + d\q v\q + v + v\q + v\qq=0\\
    a^2 u + a b v + a u + a v\q + b v + b v\qq + u + v\q + v\qqq=0\\
    a c u + a\q c u + a\q d v + a\q v\qqq + b c v + b\q v\q + c v\q + d v\qq + v\q + v\qq + v\qqq=0\\
    a ^{2q} c u + a ^{2q} d v + a ^{2q} v\qqq + a\q b\q v\q + a\q c u + a\q d v + a\q v\qq + a\q v\qqq + b\q v\q + b\q v\qqq + c u + d v +
        v + v\qqq=0.
        \end{cases}
    \end{equation*}
We divide the proof in several cases.

\textbf{-Case }$a^2+a+1\neq 0$

From the second equation we derive $u$, and substituting in the first and third we obtain
\begin{eqnarray}
    (a^2 c\q d + a^2 + a b  c^{q+1} + a c\q d + a + b  c^{q+1} + b + c\q d + 1)v+(
a^2 d\q + a  c^{q+1} + a d\q +  c^{q+1} + d\q + 1)v\q+\nonumber\\\label{tt4}(a^2 + a b + a + b  c^{q+1} + 1)v\qq+(a^2 c\q + a c\q + a +  c^{q+1} + c\q)v\qqq=0,
\end{eqnarray}
\begin{eqnarray}
    (a^{q+2}d + a^{q+1} bc + a^{q+1}  d + a\q  bc + a\q d + bc )v+(a^2b\q  + a^2 + a^{q+1} c + ab\q  + a + a\q c + b\q  + c + 1 )v\q+\nonumber\\\label{tt6}(a^2d + a^2 + abc + ad + a + a\q  bc + d + 1 )v\qq+(a^{q+2} + a^2 + a^{q+1}  + ac + a + a\q  c + a\q  + 1 )v\qqq=0.
\end{eqnarray}

Consider that if any of these coefficients is non-zero, then one of these equations would impose a limit of $q^3$ solutions $(u,v)$, as per the preceding observation.

Let us assume for the sake of contradiction that all of these coefficients vanish. It is worth noting that $c^{q+1}\neq a$, as otherwise $a^2 + ab + a + bc^{q+1} + 1\neq 0$. Similarly, observe that $c\neq 0$, as otherwise $a^2c^{q}+ac^{q}+a+c^{q+1}+c^{q}$ would equal $a$, which contradicts the condition $a^2 + ab + a + bc^{q+1} + 1\neq 0$. Hence, we deduce $c=(a^2c^{q}+ac^{q}+a+c^{q})/c^{q}$. We can substitute $c$ in the coefficient of $v\qqq$ of the polynomial $(\ref{tt6})$
and obtain
\begin{equation*}
    a^3c^{q} +a^2+ a^{q+1} + c^q=0.
\end{equation*}
Note that in this case $a^3=1$ if and only if $a=1$. Substituting $a=1$ in the coefficient of $v\qqq$ in $(\ref{tt4})$, we obtain $c^{q+1}+c^q+1=0$, and raising it to the $q$-power we also have $c^{q+1}+c+1=0$. Summing the two equations we obtain $c\in \F_q$, and so $c^2+c+1=0$, that implies $c\in \F_4\setminus\F_2$, a contradiction observing that $\F_{q^2}\cap\F_4=\F_2$. So let $a\neq 1$. We obtain $c^q=(a^2+a^{q+1})/(a^3+1)$, so also $c=(a^{2q}+a^{q+1})/(a^{3q}+1)$. Substituting it in the coefficient of $v\qqq$ in  $(\ref{tt6})$
we obtain
\begin{equation*}
    a(a^{2q+1}+\omega a^{q+1}+\omega a+\omega a^{2q}+\omega a^q+1)(a^{2q+1}+\omega^2 a^{q+1}+\omega^2 a+\omega^2 a^{2q}+\omega^2 a^q+1)=0.
\end{equation*}
 Obviously $a\neq 0$. Our aim is to prove that the two other factors are nonzero. 
 
 Assume $a^{2q+1}+\omega a^{q+1}+\omega a+\omega a^{2q}+\omega a^q+1=0$. Raising it to the $q$-power one gets $a^{q+2}+\omega^2 a^{q+1}+\omega^2 a\q+\omega^2 a^{2}+\omega^2 a+1=0$. Combining the two equations, we obtain 
 \begin{equation*}
     (a\q+\omega^2)(a+\omega)(a+a\q)=0.
 \end{equation*}
Note that $a\neq \omega,\omega^2$ since $a^2+a+1\neq 0$, so we obtain $a\in\F_q$. Substituting it in $a^{2q+1}+\omega a^{q+1}+\omega a+\omega a^{2q}+\omega a^q+1=0$ we obtain $a^3=1$, a contradiction. 

Assume now $a^{2q+1}+\omega^2 a^{q+1}+\omega^2 a+\omega^2 a^{2q}+\omega^2 a^q+1=0$. Raising it  to the $q$-power also $a^{q+2}+\omega^2 a^{q+1}+\omega^2 a\q+\omega^2 a^{2}+\omega^2 a+1$. Analogously to the previous case we obtain a contradiction, since summing the two we obtain
\begin{equation*}
     (a\q+\omega)(a+\omega^2)(a+a\q)=0.
 \end{equation*}
\textbf{-Case }$a=\omega$
 
The system reads
\begin{equation*}
        \begin{cases}
\omega u + b v + c^{q+1} u + c\q d v + c\q v\qqq + d\q v\q + v + v\q + v\qq=0\\
     \omega^2 b v+ \omega^2 v\q + b v\qq + v\qqq=0\\
    cu+ \omega^2 d v + \omega^2 v\qqq + b c v + b\q v\q + c v\q + d v\qq + v\q + v\qq + v\qqq=0\\
       v+ \omega b\q v\q + \omega^2 v\qq + b\q v\qqq=0.
        \end{cases}
    \end{equation*}
We have
\begin{equation*}
   0=b\q(\omega^2 b v+ \omega^2 v\q + b v\qq + v\qqq)+(v+ \omega b\q v\q + \omega^2 v\qq + b\q v\qqq)=(\omega^2b^{q+1}+1)v+b\q v\q+(b^{q+1} + \omega^2)v\qq.
\end{equation*}
Since this polynomial equation cannot vanish,  we have at most $q^2$ solutions $v$. 
The coefficients of $u$ in the first and third equations are $\omega+c^{q+1}$ and $c$, respectively. Since they cannot both vanish simultaneously, the claim follows.

\textbf{-Case }$a=\omega^2$
 
This case is analogous to the previous one since the system reads
\begin{equation*}
        \begin{cases}
\omega^2 u + b v + c^{q+1} u + c\q d v + c\q v\qqq + d\q v\q + v + v\q + v\qq=0\\
     \omega b v+ \omega v\q + b v\qq + v\qqq=0\\
    cu+ \omega d v + \omega v\qqq + b c v + b\q v\q + c v\q + d v\qq + v\q + v\qq + v\qqq=0\\
       v+ \omega^2 b\q v\q + \omega v\qq + b\q v\qqq=0.
        \end{cases}
    \end{equation*}
     Proceeding as before, we obtain the claim.
    \end{proof}
\subsubsection{$\alpha=0=\beta,$ $\gamma\neq 0$}

\begin{thm}
If $\alpha=\beta=0$ and $\gamma\neq 0$ then System \eqref{eq sistema iniziale} has at most $q^2$ solutions.   
\end{thm}
\begin{proof}
First note that the condition $\gamma\neq 0$ implies $\lambda =0 = \mu.$  
Also, any solution $(u,v)$ of System \eqref{eq sistema iniziale} satisfies  
    $$\begin{cases}
    au+a\qq u\qq+a\qqqq u+a\qqqq u\qq=0\\
    bv+b\qq v\qq+b\qqqq v+b\qqqq v\qq=0\\
    cu+c\qq u\qq+c\qqqq u+c\qqqq u\qq=0 \\
    dv+d\qq v\qq+d\qqqq v+d\qqqq v\qq=0
    \end{cases}.$$
We consider two cases. 
\begin{enumerate}
    \item $b$ or $c$ in $\mathbb{F}_{q^2}$. Since $\gamma\neq 0 $, $a,d \notin \mathbb{F}_{q^2}$.
    In this case 
    $$au+a\qq u\qq+a\qqqq u+a\qqqq u\qq=0, \qquad dv+d\qq v\qq+d\qqqq v+d\qqqq v\qq=0,$$
    and thus 
    $$u^{q^2}=\frac{a+a^{q^2}}{a^{q^2}+a^{q^4}}u, \quad v^{q^{2}}=\frac{d+d^{q^2}}{d^{q^2}+d^{q^4}}v,\quad v^{q^{3}}=\frac{d^q+d^{q^3}}{d^{q^3}+d^{q^5}}v^q.$$
    Condition $F_1^{(0)}(u,v)=0$ reads
    $$(a^{q^2+1}+ a^{q^4+1}+ a + a\qqqq)u=a\qq b v + a\qq  v\q + a\qqqq b v + a\qqqq  v\q.$$
    If $a^{q^2+1}+ a^{q^4+1}+ a + a\qqqq\neq 0$, then System \eqref{eq sistema iniziale} has at most $q^2$ solutions. If $a^{q^2+1}+ a^{q^4+1}+ a + a\qqqq=0$ then $a\qq b v + a\qq  v\q + a\qqqq b v + a\qqqq  v\q=0$ which yields $v^{q}=bv$. Thus,
    $$F_2^{(0)}(u,v)=b^{q^2+q+1} v + c u + d v + u\q=0,$$
    again providing at most $q^2$ solutions for System \eqref{eq sistema iniziale}.

    \item $b,c\notin \mathbb{F}_{q^2}$.
     In this case 
    $$bv+b\qq v\qq+b\qqqq v+b\qqqq v\qq=0, \qquad cu+c\qq u\qq+c\qqqq u+c\qqqq u\qq=0.$$
    The argument is the same as in the previous  case. In particular
    $$u^{q^2}=\frac{c+c^{q^2}}{c^{q^2}+c^{q^4}}u, \quad v^{q^{2}}=\frac{b+b^{q^2}}{b^{q^2}+b^{q^4}}v,\quad v^{q^{3}}=\frac{b^q+b^{q^3}}{b^{q^3}+b^{q^5}}v^q.$$
    Condition $F_1^{(0)}(u,v)=0$ reads
    $$(a c\qq + a c\qqqq + c + c\qqqq)u=b c\qq v + b c\qqqq v + c\qq v\q + c\qqqq v\q.$$
    If $a c\qq + a c\qqqq + c + c\qqqq\neq 0$, then System \eqref{eq sistema iniziale} has at most $q^2$ solutions. If $a c\qq + a c\qqqq + c + c\qqqq=0$ then $b c\qq v + b c\qqqq v + c\qq v\q + c\qqqq v\q=0$ which yields $v^{q}=bv$. Thus,
    $$F_2^{(0)}(u,v)=b^{q^2+q+1} v + c u + d v + u\q=0,$$
    again providing at most $q^2$ solutions for System \eqref{eq sistema iniziale}.
    
\end{enumerate}

\end{proof}

 \subsubsection{$\alpha \neq 0 \neq \beta$}

 \begin{thm}
If $\alpha,\beta\neq 0$ then System \eqref{eq sistema iniziale} has at most $q^2$ solutions.
 \end{thm}
 \begin{proof}
First, observe that $a,b,c,d\notin \mathbb{F}_{q^2}$, otherwise at least one among $\alpha$ and $\beta$ would vanish.
Also from Equations \eqref{eqacu} and \eqref{eqbdv}, 
\begin{eqnarray*}
    u&=&\
    \frac{ a\qq \mu +  a\qqqq \mu + c\qq \lambda + c\qqqq \lambda}{ a c\qq +  a c\qqqq +  a\qq c +  a\qq c\qqqq +  a\qqqq c +  a\qqqq c\qq},\\
    v&=& \frac{ b\qq \mu +  b\qqqq \mu + d\qq \lambda + d\qqqq \lambda}{ b d\qq +  b d\qqqq +  b\qq d +  b\qq d\qqqq +  b\qqqq d +  b\qqqq d\qq}.
 \end{eqnarray*}
This shows that System \eqref{eq sistema iniziale} has at most $q^4$ solutions. 
Now we can consider the corresponding linear system in $\lambda,\lambda^q,\mu,\mu^q$ arising from System \eqref{eq sistema iniziale}. 
Note that the number of solutions of this linear system (in terms of $\lambda,\lambda^q,\mu,\mu^q$) is precisely the number of solutions (in $u,v$) of System \eqref{eq sistema iniziale}.

After clearing the denominators, the coefficient of $\lambda^q$ in 
$$F_1^{(0)}(u,v)=0$$
is $(d+d^{q^2})^{q^3}\alpha\beta\neq 0$. Thus we can determine $\lambda^q$ in terms of $\lambda,\mu,\mu^q$. By $F_1^{(2)}(u,v)=0$ this gives a linear equation in $\lambda,\mu,\mu^q$ whose coefficient, after clearing the denominators, in $\mu^q$ is $\alpha^2\beta^{q+2}\neq 0$. This shows that the system in  $\lambda,\lambda^q,\mu,\mu^q$, and thus  System \eqref{eq sistema iniziale}, has at most $q^2$ solutions. 
 \end{proof}
 
\subsection{Case $W$ represented by \eqref{retta4}}
 \begin{thm}
     For any choice $a,b_1,b_2,b_3\in\F_{q^6}$ the $2$-dimensional space defined by $$\begin{cases}
    au+v=0\\
    b_0u+b_2w+b_3t=0
\end{cases}$$ has at most weight 2.
 \end{thm}
 \begin{proof}
 Our aim is to prove that the system below has at most $q^2$ solutions $(u,v)$ for any $a,b_0,b_2,b_3\in\F_{q^6}$
$$
\begin{cases}
    au+v=0\\
    b_0u+b_2u\qq+b_2v\q+b_3u\q+b_3v\qqq=0\\
    u+u^{q^2}+u^{q^4}=v+v^{q^2}+v^{q^4}=0.
    \end{cases}$$
This is equivalent to 
$$
\begin{cases}
    v=au\\
    b_0u+b_2u\qq+a\q b_2u\q+b_3u\q+a\qqq b_3u\qqq=0\\
    au+a\qq u^{q^2}+a\qqqq u^{q^4}=0\\
    u+u^{q^2}+u^{q^4}=0.
    \end{cases}$$
    Note that if $a\notin\F_{q^2}$, we obtain 
    $$(a+a\qqqq)u+(a\qq+a\qqqq)u\qq=0.$$
    This limits the number of $u$ to $q^2$, and since $v$ depends on $u$, we obtain the same bound on the solutions $(u, v)$ of the system.

    Now we study the case  $a\in\F_{q^2}$. First, observe that $a=0$ is trivial. So consider $a\in\F_{q^2}^*$. The system reads 
    $$
\begin{cases}
    v=au\\
    b_0u+b_2u\qq+a\q b_2u\q+b_3u\q+a\q b_3u\qqq=0\\
    u+u^{q^2}+u^{q^4}=0.
    \end{cases}$$
In this case, the $2$-dimensional subspace defined by 
$$\begin{cases}
    v=au\\
    b_0u+b_2u\qq+b_2v\q+b_3u\q+b_3v\qqq=0
\end{cases}$$ has at most weight $3$ and is contained in $v-au=0$, that is fixed by $x\longrightarrow x^{q^2}$. So, by Theorem \ref{thm3to2} the claim follows. 
\end{proof}

\section{The Delsarte dual}

In \cite[Section 3]{CsMPZ2019}, another type of duality has been introduced. 	
	Let $U$ be an $n$-dimensional $\F_q$-subspace of a vector space $V=V(k,q^m)$, with $n>k$. By \cite[Theorems 1, 2]{LuPo2004} (see also \cite[Theorem 1]{LuPoPo2002}), there is an embedding of $V$ in $Z=V(n,q^m)$ with $Z=V \oplus \Gamma$ for some $(n-k)$-dimensional $\F_{q^m}$-subspace $\Gamma$ such that
	$U=\langle W,\Gamma\rangle_{\F_{q}}\cap V$, where $W$ is a $n$-dimensional $\F_q$-subspace of $Z$, $\langle W\rangle_{\F_{q^m}}=Z$ and $\Gamma\cap V=W\cap \Gamma=\{{ 0}\}$.
	Then the quotient space $Z/\Gamma$ is isomorphic to $V$ and under this isomorphism $U$ is the image of the $\F_q$-subspace $W+\Gamma$ of $Z /\Gamma$.
	Now, let $\beta'\colon W\times W\rightarrow\F_{q}$ be a non-degenerate bilinear form on $W$. Then $\beta'$ can be extended to a non-degenerate bilinear form $\beta\colon Z\times Z\rightarrow\F_{q^m}$.
	Let $\perp$ and $\perp'$ be the orthogonal complement maps defined by $\beta$ and $\beta'$ on the lattice of $\F_{q^m}$-subspaces of $Z$ and of $\F_q$-subspaces of $W$, respectively.
	The $k$-dimensional $\F_q$-subspace $W+\Gamma^{\perp}$ of the quotient space $Z/\Gamma^{\perp}$  will be denoted by $\bar U$ and we call it the \textit{Delsarte dual}{} of $U$ with respect to $\beta'$. By \cite[Remark 3.7]{CsMPZ2019}, up to $\mathrm{GL}(n,q)$-equivalence,  the Delsarte dual of an $n$-dimensional $\F_q$-subspace does not depend on the choice of the non-degenerate bilinear form on $W$. By \cite[Theorem 3.3]{CsMPZ2019}, the Delsarte dual of a maximum $2$-scattered of a vector space $V=V(4,q^6)$ is a maximum $2$-scattered subspace of a vector space isomorphic to $V$.

	\begin{prop}
	The $\F_q$-subspace $U$ and its Delsarte dual $\bar U$ are $\mathrm{GL}(4,q^6)$-equivalent
	\end{prop}
	\begin{proof}
	Let $U=U_1$ be the maximum $2$-scattered of $V=\F_{q^6}^4$, $q=2^{2h+1}$, defined in \eqref{formU}.
	Using the notations above we can embed $V$ in $Z=\F_{q^6}^8$ in such a way that
	\begin{align*}
	    &V=\left\{\left(Y_0,Y_1,Y_2,Y_3,0,0,0,0\right): Y_0,Y_1,Y_2,Y_4\in\F_{q^4}\right\},\\
	    &W=\left\{(x,y,x^q,y^q,x^{q^2},y^{q^2},x^{q^3},y^{q^3}): x,y\in\F_{q^6},\ \Tr_{q^6/q^2}(x)=\Tr_{q^6/q^2}(y)=0\right\},\\
	    &\Gamma=\left\{\left(0,0,X_2,X_3,X_3,X_5,X_6,X_2\right): X_2,X_3,X_5,X_6\in\F_{q^4}\right\}.
	\end{align*}

Note that $W$ is an $8$-dimensional $\F_q$-subspace with $\langle W\rangle_{\F_{q^6}}=Z$. Indeed  $W=W_1\oplus W_2$, where
\begin{align*}
    &W_1=\left\{(x,0,x^q,0,x^{q^2},0,x^{q^3},0): x\in\F_{q^6},\ \Tr_{q^6/q^2}(x)=0\right\},\\
    &W_2=\left\{(0,y,0,y^q,0,y^{q^2},0,y^{q^3}): y\in\F_{q^6},\ \Tr_{q^6/q^2}(y)=0\right\}.
\end{align*}
Also, $\dim_{\F_{q^6}} \langle W_1\rangle=\dim_{\F_{q^6}}\langle W_2\rangle=4$. This is a consequence of the fact that the set $T=\{x\in\F_{q^6}\colon \Tr_{q^6/q^2}(x)=0\}$ is a $4$-dimensional $\F_q$-subspace of $\F_{q^6}$. So let $\{t_1,t_2,t_3,t_4\}$ be an $\F_q$-basis of it. The vectors
\[\left\{(t_1,0,t_1^q,0,t_1^{q^2},0,t_1^{q^3},0),(t_2,0,t_2^q,0,t_2^{q^2},0,t_2^{q^3},0), (t_3,0,t_3^q,0,t_3^{q^2},0,t_3^{q^3},0), (t_4,0,t_4^q,0,t_4^{q^2},0,t_4^{q^3},0)\right\}\] and 
\[\left\{(0,t_1,0,t_1^q,0,t_1^{q^2},0,t_1^{q^3}),(0,t_2,0,t_2^q,0,t_2^{q^2},0,t_2^{q^3}), (0,t_3,0,t_3^q,0,t_3^{q^2},0,t_3^{q^3}), (0,t_4,0,t_4^q,0,t_4^{q^2},0,t_4^{q^3})\right\}\] 
are $\F_{q^6}$-independent.

In fact, the rank of the four vectors 
$$(t_1,0,t_1^q,0,t_1^{q^2},0,t_1^{q^3},0),(t_2,0,t_2^q,0,t_2^{q^2},0,t_2^{q^3},0), (t_3,0,t_3^q,0,t_3^{q^2},0,t_3^{q^3},0), (t_4,0,t_4^q,0,t_4^{q^2},0,t_4^{q^3},0)$$ is four if and only if the determinant of 
$$
M(t_1,t_2,t_3,t_4):=\begin{pmatrix}
t_1&t_1^q&t_1^{q^2}&t_1^{q^3}\\
t_2&t_2^q&t_2^{q^2}&t_2^{q^3}\\
t_3&t_3^q&t_3^{q^2}&t_3^{q^3}\\
t_4&t_4^q&t_4^{q^2}&t_4^{q^3}\\
\end{pmatrix}
$$
does not vanish.
The matrix above is a so-called square Moore matrix, the $q$-analog of the Vandermonde matrix introduced by Moore \cite{Moore}. It is well known that $\det(M(t_1,t_2,t_3,t_4))=0$ if and only if $t_1,t_2,t_3,t_4$ are $\mathbb{F}_q$-linearly dependent. Since $\left\{t_1,t_2,t_3,t_4\right\}$ are $\F_q$-independent, $\dim_{\F_{q^6}} \langle W_1\rangle=\dim_{\F_{q^6}}\langle W_2\rangle=4$.
		Straightforward computations show that $U$ is $\mathrm{GL}(4,q^6)$ equivalent to 
		\[\langle W,\Gamma\rangle_{\F_{q}}\cap V=\left\{(x,y,x^q+y^{q^3},y^q+x^{q^2},0,0,0,0): x,y\in\F_{q^6},\ \Tr_{q^6/q^2}(x)=\Tr_{q^6/q^2}(y)=0\right\}.\]

    Let $({\bf x},{\bf y}):=(x,y,x^q,y^q,x^{q^2},y^{q^2},x^{q^3},y^{q^3})$, with $\Tr_{q^6/q^2}(x)=\Tr_{q^6/q^2}(y)=0$, and  consider the bilinear form $\beta'$ on $W$ defined as
	$$\beta'(({\bf x},{\bf y}),({\bf u},{\bf v}))=\Tr_{q^6/q}(xv-uy).$$
	Then $\beta'$ can be extended to the non-degenerate bilinear form $\beta$ of $Z$ defined as:
	
	\[\beta\left({\mathbf X}, {\mathbf Y}\right)=X_0Y_4+X_4Y_0+X_1Y_5+X_5Y_1+X_2Y_6+X_6Y_2+X_3Y_7+X_7Y_3,\]
	where ${\mathbf X}=(X_0,X_1,X_2,X_3,X_4,X_5,X_6,X_7)$ and ${\mathbf Y}=(Y_0,Y_1,Y_2,Y_3,Y_4,Y_5,Y_6,Y_7)$.
	Then 
	$$\Gamma^\perp=\left\{\left(Z_0,0,0,Z_3,Z_4,Z_5,Z_3,Z_0\right): Z_0,Z_3,Z_4,Z_5\in\F_{q^6}\right\}.$$
	Hence, the Delsarte dual $\bar U=\langle W,\Gamma^\perp\rangle_{\F_{q}}\cap \Delta$, where $$\Delta=\left\{\left(X_0,X_1,X_2,X_3,0,0,0,0\right): X_0,X_1,X_2,X_3\in\F_{q^6}\right\},$$ turns out to be
	\begin{align*}
	  \bar U&=\left\{(x+y^{q^3},y,x^q,y^{q}+x^{q^3},0,0,0,0): x,y\in\F_{q^6}, \Tr_{q^6/q^2}(x)=\Tr_{q^6/q^2}(y)=0\right\}\\
	  &=\left\{(z^q+z^{q^3}+t^{q^3},t,z,t^q+z^{q^2},0,0,0,0): t,z\in\F_{q^6}, \Tr_{q^6/q^2}(z)=\Tr_{q^6/q^2}(t)=0\right\}.
	\end{align*}
    Since
    \[\begin{pmatrix}
1 & 1 & 1 & 0\\
1 & 1 & 0 & 1\\
1 & 1 & 1 & 1\\
1 & 0 & 1 & 0
\end{pmatrix}
\begin{pmatrix}
x\\
y\\
x^{q^2}+y^{q}\\
x^{q}+y^{q^3}
\end{pmatrix}=
\begin{pmatrix}
z\\
t\\
z^{q^{2}}+t^{q}\\
z^{q}+z^{q^3}+t^{q^{3}}
\end{pmatrix},
\]
with $\Tr_{q^6/q^2}(x)=\Tr_{q^6/q^2}(y)=\Tr_{q^6/q^2}(z)=\Tr_{q^6/q^2}(t)=0$, the proof is now complete.
\end{proof}

\section{Rank-metric codes}

{\it Rank-metric} codes were originally introduced by Delsarte in the late $70$'s \cite{Delsarte}, and then resumed a few years later by Gabidulin in \cite{Gabidulin}.

Let $m,n \in \mathbb{N}$ be two positive integers such that $m,n \geq 2$, and let $q$ be a prime power. Let $\F_{q^m}^n$ be the vector space of dimension $n$ over the Galois field $\F_{q^m}$. 
Let consider $v=(v_1,v_2,...,v_n) \in \F_{q^m}^n$, the \textit{rank weight} of $v$ is defined as
$$\omega_{\mathrm{rk}}(v)=\dim \langle v_1,v_2,\ldots,v_n\rangle_{\F_q}.$$
A rank-metric code ${\cal C} \subseteq \F_{q^m}^n$ of {\it length} $n$, is a subset of $\F_{q^m}^n$ considered as a vector space endowed with the metric defined by the map $$d(v,w)= \omega_{rk}(v-w),$$
where $v$ and $w \in \F_{q^m}^n.$ Elements of ${\cal C}$ are called {\it codewords}. A {\it linear} rank-metric code ${\cal C}$ is an $\F_{q^m}$-subspace of $\F_{q^m}^n$ endowed with
the rank metric. If ${\cal C} \subseteq \F_{q^m}^n$ is a linear rank-metric code, the minimum distance between two distinct codewords of ${\cal C}$ is $$d=d({\cal C})=\min \{\omega_{rk}(v) \, | \, v \in \cal C \setminus \{{\bf 0}\}\}.$$ 

\noindent A linear rank-metric code ${\cal C} \subseteq \F_{q^m}^n$ of length $n$, dimension $k$ and minimum distance $d$, is referred in the literature as to an $[n,k,d]_{q^m/q}$ code, or as to an $[n,k]_{q^m/q}$ code, depending on whether the minimum distance is known or not.  
These parameters are related by an inequality, which is known as the Singleton-like bound. Precisely, if $\cC$ is an $[n,k,d]_{q^m/q}$ code, then 
\begin{equation}\label{singleton-bound}
    mk \leq \min\{m(n - d + 1), n(m - d + 1)\}, 
\end{equation}
see \cite{Delsarte}.

Codes attaining this bound with equality are called {\it maximum rank distance (MRD) codes}, and they are considered to be optimal, due to their largest possible error-correction capability. From the classification of $\F_{q^m}$-linear isometry (see \cite{Berger}), we say that two $[n,k,d]_{q^m/q}$ codes ${\cal C} _1, {\cal C} _2$ are {\it (linearly) equivalent} if there exist $A\in \mathrm{GL}(n,q)$ and $a\in\F_{q^m}^*$ such that \begin{equation*}
{\cal C} _2=a{\cal C} _1\cdot A=\{avA \,:\, v \in {\cal C}_1\}.
\end{equation*}

The geometric counterpart of non-degenerate rank-metric codes (that is,
the columns of any generator matrix of C are $\F_q$-linearly independent) are the
$q$-systems. Let $U$ be an $\F_q$-subspace of $\F_{q^m}^k$ and let $H$ be an $\F_{q^m}$-subspace of $\F_{q^m}^k$. The \textit{weight} of $H$ in $U$ is $\mathrm{wt}_{U}(H)=\dim_{\F_q}(H\cap U)$. Assume now that $U$ has dimension $n$ over $\F_q$. 

We say that $U$ is an $[n,k,d]_{q^m/q}$ \textit{system} if $\langle U \rangle_{\F_{q^m}}=\F_{q^m}^k$ and 
\begin{equation*}
d=\,n-\mathrm{max}\{\mathrm{wt}_{U}(H) \,:\, H \subseteq \F_{q^m}^k \textnormal{ with } \dim_{\F_{q^m}}(H)=k-1 \}.
\end{equation*}

More generally, for each $1 \leq \rho \leq k-1$, the parameters

\begin{equation*}
d_{\rho}=\,n-\mathrm{max}\{\mathrm{wt}_{U}(H) \,:\, H \subseteq \F_{q^m}^k \textnormal{ with } \dim_{\F_{q^m}}(H)=k-\rho \},
\end{equation*}
are known as the {\it $\rho$-generalized rank weight} of the system $U$; see \cite[Definition 4]{Randrianarisoa2020geometric}.

 As before, if the parameter $d$ is not relevant, we will write that $U$ is an $[n,k]_{q^m/q}$ system. Furthermore, when none of the parameters is relevant, we will generically refer to $U$ as to a $q$-system.

 Two $[n,k,d]_{q^m/q}$ systems $U_1,U_2$ are (linearly) equivalent if there exists $A\in \mathrm{GL}(k,q^m)$ such that 
 \begin{equation*}
 U_1\cdot A:=\{uA \,:\, u \in U_1\}=U_2.
 \end{equation*}

In \cite{Randrianarisoa2020geometric}, a one-to-one correspondence between the equivalence class of non-degenerate codes  and the equivalence class of the $q$-systems, is established. Moreover, under this correspondence that associates an $[n,k,d]_{q^m/q}$ code $\mathcal{C}$ to an $[n,k,d]_{q^m/q}$ system $U$, codewords of $\mathcal{C}$ of rank weight $w$ correspond to $\mathbb{F}_{q^m}$-hyperplanes $H$ of $\mathbb{F}_{q^m}^k$ with $\mathrm{wt}_U(H)=n-w$. Also, if $\mathcal{C}$ is the $[n,k,d]_{q^m/q}$ code associated to $U$, the $\rho$-generalized weight of $\mathcal{C}$ is simply defined as the $\rho$-generalized weight of the associated system $U$. It is easy to see that the first generalized weight $d_1$ of $\mathcal{C}$, coincides with its minimum distance $d$.

\begin{comment}
    
In \cite{LiaLongobardiMarinoTrombetti} the $\rho$-generalized weight of an $\F_q$-subspace $V\subset \mathbb{F}_{q^m}^k$ of rank $n$ over $\mathbb{F}_q$ is defined as
\begin{equation*}
    d_\rho(V):=n-\max\{\dim_{\F_q}(V\cap H): H\subset \mathbb{F}_{q^m}^k \text{ with } \dim_{\F_{q^m}}(H)=k-\rho \}.
\end{equation*}
From \cite{martinez2016similarities} we have the following bound
\begin{equation*}
     d_\rho(V)\leq\min\{n-k+\rho,\rho m,\frac{m}{n}(n-k)+m(\rho-1)+1\}.
\end{equation*}

\end{comment}

We have the following definitions; see \cite{MarinoNeriTrombetti}.

\begin{definition}\label{defin}
     Let $\rho$ be a positive integer.  An $[n,k]_{q^m/q}$ code $\mathcal{C}$ is  $\rho$-MRD if 
 $d_{\rho}(\mathcal{C})=n-k+\rho$.
\end{definition}

Notice that if an $[n,k]_{q^m/q}$ code $\mathcal{C}$ is $\rho$-MRD for some $\rho$, then it is also $\rho'$-MRD for every $\rho\leq \rho'\leq k$.

From \cite[Remark 2.12]{MarinoNeriTrombetti}, a $1$-MRD code is also an MRD code, whereas an MRD code is not necessarily $1$-MRD. Also, for $\rho\geq 2$, there is very little known about $\rho$-MRD codes which are not $1$-MRD.

\begin{definition}\label{defin1}
 An $[n,k]_{q^m/q}$ code $\mathcal{C}$  is called \textit{near MRD} if  $d(\mathcal{C})=n-k$ and $d_{\rho}(\mathcal{C})=n-k+\rho$ for every $2\leq \rho\leq k$.
\end{definition}

\begin{comment}
     \begin{definition}
         A RD code $\mathcal{C}$ associated to $V$ is said to be a $\rho$-MRD code if $d_\rho(V)=n-k+\rho$.
     \end{definition}
     Note that if $\mathcal{C}$ is $\rho$-MRD then is also $\rho'$-MRD for every $\rho\leq\rho'\leq k$.
It's easy to see that a $1$-MRD code is also an MRD code, but the opposite is not true. Therefore, examples of MRD codes that are 2-MRD codes but not 1-MRD codes are of particular interest. 

\end{comment}

\begin{prop}
    Let $\mathcal{C}$ be an $[8,4,4]_{q^6/q}$ MRD code associated with $U_s$ as in \eqref{formU}. Then $d_2=6$.
\end{prop}
\begin{proof}
    The value of $d_2$ follows from the fact that 
    $U_s$ is $2$-scattered.
\end{proof}

From the previous proposition and Definitions \ref{defin} and \ref{defin1} we get the following result.

\begin{cor}
    Let $\mathcal{C}$ be an $[8,4,4]_{q^6/q}$ MRD associated with $U_s$ as in \eqref{formU}. Then
$\mathcal{C}$ is a $2$-MRD code but not a $1$-MRD code.    
\end{cor}

\begin{rem}
    \rm{
The code corresponding to $U_s$ as described in \eqref{formU} resolves the question raised in \cite[Section 5.2]{MarinoNeriTrombetti} regarding the search for an $[8,4,4]_{q^6/q}$ near Maximum Rank Distance (MRD) code.
   }
\end{rem}

\begin{rem}
    \rm{
Since there exist maximum $2$-scattered $\F_q$-subspaces in an $\F_{q^6}$-space of dimension $3$ and $4$, by \cite[Theorem 2.5]{CsMPZ2019}, there exist maximum $2$-scattered $\F_q$-subspaces in $V(r,q^6)$, with $r\geq 3$ and $r\ne 5$, $q=2^h$, with $h$ odd. Also, if $\mathcal{C}$ is the associated $[2r,r,4]_{q^6/q}$ MRD code then $d_{r-2}(\mathcal{C})=2r-2$, i.e. it is an $(r-2)$-MRD but not a $1$-MRD.}
\end{rem}

\begin{comment}
    
We can show that the codes associated to $U$ have such property.
 From the main Theorem, we have that $d_2(U)=8-2=6$. We also have that $n-k+2=8-4+2=6$, so the codes associated to $U$ are $2$-MRD codes.
 For $\rho=1$ we have that $n-k+\rho=5$, but $d_1(U)\leq\min\{5,6,4\}$, so the codes associated to $U$ are not $1$-MRD codes.

\end{comment}

\section*{Acknowledgements}
The authors thank the Italian National Group for Algebraic and Geometric Structures and their Applications (GNSAGA—INdAM)
which supported the research. 

\section*{Declarations}
{\bf Conflicts of interest.} The authors have no conflicts of interest to declare that are relevant to the content of this
article.


\begin{thebibliography}{100}



\bibitem{BBL2000}
{\sc S. Ball, A. Blokhuis and M. Lavrauw:}
Linear $(q+1)$-fold blocking sets in $\PG(2,q^4)$,
\emph{Finite Fields Appl.} {\bf 6 (4)} (2000), 294--301.

\bibitem{BaCsMT2020}
{\sc D. Bartoli, B. Csajb\'ok, G. Marino and R. Trombetti:}
Evasive subspaces,
{\it Journal of Combinatorial Designs} {\bf 29} (2021), 533--551.


\bibitem{articolosequenze-new} {\sc D. Bartoli, A. Giannoni and G. Marino:} 
New scattered subspaces in higher dimensions, \href{https://	arXiv:2402.15223}{	arXiv:2402.15223}.


\bibitem{BGMP2015}
{\sc D. Bartoli, M. Giulietti, G. Marino and O. Polverino:}
Maximum scattered linear sets and complete caps in Galois spaces,
\emph{Combinatorica} {\bf 38}(2) (2018), 255--278.


\bibitem{articolosequenze} {\sc D. Bartoli, G. Marino, A. Neri, and L. Vicino:} 
Exceptional scattered sequences, \href{https://	 arXiv:2211.11477}{	 arXiv:2211.11477}

    
\bibitem{BZ2018}
{\sc D. Bartoli, Y. Zhou:}
Exceptional scattered polynomials,
J. Algebra {\bf 509} (2018), 507--534.

\bibitem{Berger}
{\sc T.P. Berger:}
Isometries for rank distance and permutation group of Gabidulin codes, IEEE Trans. Inf. Theory {\bf 49 (11)} (2003), 3016--3019.


\bibitem{BL2000}
{\sc A. Blokhuis and M. Lavrauw:}
Scattered spaces with respect to a spread in $\mathrm{PG}(n,q)$,
\emph{Geom.\ Dedicata} {\bf 81} (2000), 231--243.



\bibitem{CSMPZ2016}
{\sc B. Csajb\'ok, G. Marino, O. Polverino and F. Zullo:}
Maximum scattered linear sets and MRD-codes,
{\it J.\ Algebraic\ Combin.} {\bf 46} (2017), 1--15.

\bibitem{CsMPZ2019}
		{\sc B. Csajb\'ok, G. Marino, O. Polverino and F. Zullo:}
		Generalising the scattered property of subspaces, {\it Combinatorica} {\bf 41} (2021), 237--262.
  
\bibitem{Delsarte}
{\sc P. Delsarte:}
Bilinear forms over a finite field, with applications to coding theory,
{\it J.\ Combin.\ Theory Ser.\ A} {\bf 25} (1978), 226--241.


\bibitem{Gabidulin}
{\sc E. Gabidulin:}
Theory of codes with maximum rank distance,
\emph{Problems of information transmission}, {\bf 21(3)} (1985), 3--16.

\begin{comment}		
\bibitem{LiaLongobardiMarinoTrombetti}
{\sc S. Lia, G. Longobardi, G. Marino and R. Trombetti:}
Short rank-metric codes and scattered subspaces,
\href{https://arxiv.org/abs/1906.10590v2}{arXiv:2005.08401}.		
\end{comment}

\begin{comment}
\bibitem{LN}
{\sc R. Lidl, H. Niederreiter}: {\em Finite Fields}, Cambridge
University Press, 1997.    
\end{comment}


\bibitem{LuPoPo2002}
{\sc G. Lunardon, P. Polito and O. Polverino:} A geometric characterisation of linear k-blocking sets, {\it J. Geom.} {\bf 74 (1-2)} (2002), 120--122.

\bibitem{LuPo2004}
{\sc G. Lunardon and O. Polverino:}
Translation ovoids of orthogonal polar spaces,
{\it Forum Math.} {\bf 16} (2004), 663--669.

\bibitem{MarinoNeriTrombetti}{\sc G. Marino, A. Neri, and R. Trombetti:} 
Evasive subspaces, generalized rank weights and near MRD codes,
{\it Discrete Mathematics} {\bf 346 (12)} (2023), 113605.

\begin{comment}
    \bibitem{martinez2016similarities}{\sc U.~Mart{\'\i}nez-Pe{\~n}as:} 
On the similarities between generalized rank and {H}amming weights
  and their applications to network coding,
{\it IEEE Transactions on Information Theory} {\bf
62(7)} (2016), 4081--4095.
\end{comment}


\bibitem{Moore}
{\sc E.H. Moore:} 
A two-fold generalization of Fermat's theorem, 
{\it Bull. Am. Math. Soc.} {\bf 2}(7) (1896), 189--199.

\begin{comment}
\bibitem{Polverino}
{\sc O. Polverino:}
Linear sets in finite projective spaces,
\emph{Discrete Math.} {\bf 310(22)} (2010), 3096--3107.
\end{comment}

\bibitem{Randrianarisoa2020geometric}
{\sc T. H. Randrianarisoa:} A geometric approach to rank metric codes and a classification of constant weight codes, {\it Des. Codes Cryptogr.}
{\bf 88(7)} (2020), 1331--1348.


\bibitem{sheekey_new_2016}			   
{\sc J. Sheekey:}
 A new family of linear maximum rank distance codes, Adv. Math. Commun. {\bf 10} (2016), 475--488.
 

\bibitem{ShVdV}
{\sc J. Sheekey and G. Van de Voorde:}
Rank-metric codes, linear sets and their duality,
{\it Des. Codes Cryptogr.} {\bf 88} (2020), 655–-675.




\bibitem{zini2021scattered}
{\sc G.~Zini and F.~Zullo:}
Scattered subspaces and related codes.
{\it Designs, Codes and Cryptography}, 89 (2021) 1853--1873.



\end{thebibliography}
\end{document}